\documentclass[twoside]{article}
\usepackage[a4paper]{geometry}
\usepackage[latin1]{inputenc} 
\usepackage[T1]{fontenc} 
\usepackage{RR}
\usepackage{hyperref}

\usepackage{amsmath}
\usepackage{amsfonts}
\usepackage{amssymb}
\usepackage{amsthm}
\usepackage{enumerate}
\usepackage{color}

\theoremstyle{plain}

\newtheorem{existadj}{Lemma}

\newtheorem{mult1}[existadj]{Lemma}
\newtheorem{mult2}[existadj]{Lemma}
\newtheorem{polyhedricity}[existadj]{Lemma}

\newtheorem{absecond}[existadj]{Lemma}
\newtheorem{density}[existadj]{Lemma}
\newtheorem{total}[existadj]{Lemma}
\newtheorem{totalbis}[existadj]{Lemma}
\newtheorem{derivative}[existadj]{Lemma}
\newtheorem{gron}[existadj]{Lemma}

\newtheorem{first}[existadj]{Theorem}
\newtheorem{second}[existadj]{Theorem}
\newtheorem{firsteasy}[existadj]{Theorem}
\newtheorem{muep}[existadj]{Lemma}
\newtheorem{firstreduced}[existadj]{Lemma}

\newtheorem{inj}[existadj]{Lemma}

\newtheorem{sufficient}[existadj]{Theorem}
\newtheorem{invers}[existadj]{Lemma}
\newtheorem{interest}[existadj]{Lemma}

\newtheorem{coupe}[existadj]{Lemma}
\newtheorem{bigresult}[existadj]{Lemma}
\newtheorem{raccords}[existadj]{Lemma}
\newtheorem{resol}[existadj]{Lemma}
\newtheorem{rempl}[existadj]{Lemma}
\newtheorem{rempl2}[existadj]{Lemma}
\newtheorem{out}[existadj]{Lemma}
\newtheorem{isombv}[existadj]{Lemma}
\newtheorem{ipp}[existadj]{Lemma}
\newtheorem{nogap}[existadj]{Corollary}

\theoremstyle{definition}

\newtheorem{lagrangemult}[existadj]{Definition}
\newtheorem{order}[existadj]{Definition}
\newtheorem{constraints}[existadj]{Definition}
\newtheorem{qualif}[existadj]{Definition}
\newtheorem{legendre}[existadj]{Definition}
\newtheorem{truncation}[existadj]{Definition}

\newtheorem{link}[existadj]{Remark}
\newtheorem{assump}[existadj]{Remark}
\newtheorem{exlemma}[existadj]{Remark}
\newtheorem{huu}[existadj]{Remark}
\newtheorem{reducible}[existadj]{Remark}
\newtheorem{muepsuite}[existadj]{Remark}
\newtheorem{finalrmk}[existadj]{Remark}
\newtheorem{qualifrem}[existadj]{Remark}

\numberwithin{equation}{section}

\RRNo{7961}
\RRdate{May 2012}
\RRauthor{
J.~Fr\'ed\'eric Bonnans\thanks{Inria Saclay \& CMAP, Ecole Polytechnique, 91128 Palaiseau, France (frederic.bonnans@inria.fr)}%
  \and
Constanza de la Vega\thanks{IMAS-CONICET \& Departamento de Matem\'atica, UBA, Buenos Aires, Argentina (csfvega@dm.uba.ar)}%
\and Xavier Dupuis\thanks{Inria Saclay \& CMAP, Ecole Polytechnique, 91128 Palaiseau, France (xavier.dupuis@cmap.polytechnique.fr)}
}
\authorhead{J.~Fr\'ed\'eric Bonnans \& Constanza de la Vega \& Xavier Dupuis}
\RRtitle{Conditions d'optimalit\'e du premier et second ordre pour des probl\`emes de commande optimale d'\'equations int\'egrales
avec contraintes sur l'\'etat}
\RRetitle{First and second order optimality conditions for optimal control problems of state constrained integral equations}
\titlehead{Optimal control of state constrained integral equations}
\RRnote{Published as Inria report number 7961, May 2012; hal.inria.fr/hal-00697504}
\RRnote{
 The research leading to these results has received funding from the EU 7th Framework Programme (FP7-PEOPLE-2010-ITN), under GA number 264735-SADCO.
The first and third author also thank the MMSN (Mod\'elisation Math\'ematique et Simulation Num\'erique) Chair (EADS, Inria and Ecole Polytechnique) for its support.
}

\RRresume{On s'int\'eresse dans cet article \`a des probl\`emes de commande optimale d'\'equations int\'egrales,
avec contraintes sur l'\'etat initial-final ainsi que sur l'\'etat \`a chaque instant.
L'ordre d'une contrainte sur l'\'etat \`a chaque instant est d\'efini dans le cadre d'une dynamique int\'egrale,
et on consid\`ere ici des contraintes d'ordre quelconque. On obtient des conditions d'optimalit\'e n\'ecessaires
du premier et second ordre, ainsi que des conditions suffisantes du second ordre.
}

\RRabstract{
This paper deals with optimal control problems of integral equations, with initial-final and running state constraints.
The order of a running state constraint is defined in the setting of integral dynamics, and we work here with constraints of arbitrary high orders.
First and second-order necessary conditions of optimality are obtained, as well as second-order sufficient conditions.}
\RRmotcle{commande optimale, \'equations int\'egrales, contraintes sur l'\'etat, conditions d'optimalit\'e du second ordre}
\RRkeyword{optimal control, integral equations, state constraints, second-order optimality conditons}
\RRprojets{Commands}
\RCSaclay 

\newcommand{\ud}{\mathrm{d}}
\newcommand{\e}{\varepsilon}
\newcommand{\p}{\varphi}
\newcommand{\R}{\mathbb R}

\def\D{\widehat{D}}

\def\itr{\item[$\triangleright$]}

\DeclareMathOperator*{\dist}{dist}
\DeclareMathOperator*{\supp}{supp}
\DeclareMathOperator*{\ri}{ri}

\def\uh{\hat{u}}

\def\xh{\hat{x}}

\def\bb{\bar{b}}

\def\hb{\bar{h}}

\def\ub{\bar{u}}
\def\vb{\bar{v}}
\def\wb{\bar{w}}

\def\yb{\bar{y}}
\def\zb{\bar{z}} 

\def\bt{\tilde{b}}

\def\ut{\tilde{u}}
\def\vt{\tilde{v}}

\def\yt{\tilde{y}}
\def\zt{\tilde{z}}

\def\cP{{\cal P}}

\def\A{\mathcal{A}}

\def\I{\mathcal{I}}
\def\J{\mathcal{J}}

\def\M{\mathcal{M}}

\def\SS{\mathcal{S}}
\def\T{\mathcal{T}}
\def\U{\mathcal{U}}
\def\V{\mathcal{V}}

\def\Y{\mathcal{Y}}
\def\Z{\mathcal{Z}}

\def\dist{\mathop{\rm dist}}

\def\supp{\mathop{\rm supp}}

\def\1B{{\bf  1}}

\newcommand{\NN}{\mathbb{N}}

\newcommand\be{\begin{equation}}
\newcommand\ee{\end{equation}}
\newcommand\ba{\begin{array}}
\newcommand\ea{\end{array}}
\newcommand{\bean}{\begin{eqnarray*}}
\newcommand{\eean}{\end{eqnarray*}}
\newcommand\bal{\begin{align}}
\newcommand\eal{\end{align}}

\begin{document}
\makeRR   

\section{Introduction}

The dynamics in the optimal control problems we consider in this paper is given by an integral equation.
Such equations, sometimes called nonlinear Volterra integral equations, belong to the family of equations with memory and thus are found in many models.
Among the fields of application of these equations are population dynamics in biology and growth theory in economy:
see \cite{citeulike:8268303} or its translation in \cite{MR521933} for one of the first use of integral equations in ecology in 1927 by Volterra,
who contributed earlier to their theoretical study \cite{Volterra13};
in 1976, Kamien and Muller model the capital replacement problem by an optimal control problem with an integral state equation \cite{Kamien}.
First-order optimality conditions for such problems were known under the form of a maximum principle since Vinokurov's paper \cite{MR0218951} in 1967,
translated in 1969 \cite{MR0247559} and whose proof has been questionned by Neustadt and Warga \cite{MR0273489} in 1970.
Maximum principles have then been provided by Bakke \cite{MR0333893}, Carlson \cite{MR896645}, or more recently de la Vega \cite{MR2275355} for an optimal terminal time control problem.
First-order optimality conditions for control problems of the more general family of equations with memory are obtained by Carlier and Tahraoui \cite{MR2451792}.

None of the previously cited articles consider what we will call 'running state constraints'.
That is what Bonnans and de la Vega did in \cite{MR2739581}, where they provide Pontryagin's principle, i.e. first-order optimality conditions.
In this work we are particularly interested in second-order necessary conditions, in presence of running state constraints.
Such constraints drive to optimization problems with inequality constraints in the infinite-dimensional space of continuous functions.
Thus second-order necessary conditions on a so-called \textit{critical cone} will contain an extra term,
as it has been discovered in 1988 by Kawasaki \cite{MR941318} and generalized in 1990 by Cominetti \cite{MR1036588}, in an abstract setting.
It is possible to compute this extra term in the case of state constrained optimal control problems;
this is what is done by P\'ales and Zeidan \cite{MR2048167} or Bonnans and Hermant \cite{MR2421298,MR2504044} in the framework of ODEs.

Our strategy here is different and follows \cite{MR2779110}, with the differences that we work with integral equations
and that we add initial-final state constraints which lead to nonunique Lagrange multipliers.
The idea was already present in \cite{MR941318} and is closely related to the concept of extended polyhedricity \cite{MR1756264}:
the extra term mentioned above vanishes if we write second-order necessary conditions on a subset of the critical cone, the so-called \textit{radial critical cone}.
This motivates to introduce an auxiliary optimization problem, the \textit{reduced problem}, for which under some assumptions
the radial critical cone is dense in the critical cone.
Optimality conditions for the reduced problem are relevant for the original problem and the extra term now appears as the derivative of a new constraint in the reduced problem.
We will devote a lot of effort to the proof of the density result and we will mention a flaw in \cite{MR2779110} concerning this proof.

The paper is organized as follows.
We set the optimal control problem, define Lagrange multipliers and work on the notion of order of a running state constraint in our setting in section~\ref{ocp}.
The reduced problem is introduced in section~\ref{weak}, followed by first-order necessary conditions and second-order necessary conditions on the radial critical cone.
The main results are presented in section~\ref{strong}.
After some specific assumptions, we state and prove the technical lemma~\ref{bigresult} which is then used to strengthen the first-order necessary conditions already obtained
and to get the density result that we need. With this density result, we obtain second-order necessary conditions on the critical cone.
Second-order suficient conditions are also given in this section.
Some of the technical aspects are postponed in the appendix.

\paragraph{Notations}
We denote by $h_t$ the value of a function $h$ at time $t$ if $h$ depends only on $t$, and by $h_{i,t}$ its $i$th component if $h$ is vector-valued.
To avoid confusion we denote partial derivatives of a function $h$ of $(t,x)$ by $D_t h$ and $D_x h$.
We identify the dual space of $\mathbb R^n$ with the space $\mathbb R^{n*}$ of $n$-dimensional horizontal vectors.
Generally, we denote by $X^*$ the dual space of a topological vector space $X$.
We use $| \cdot |$ for both the Euclidean norm on finite-dimensional vector spaces and for the cardinal of finite sets,
$\| \cdot \|_s$ and $\| \cdot \|_{q,s}$ for the standard norms on the Lesbesgue spaces $L^s$ and the Sobolev spaces $W^{q,s}$, respectively.

\section{Optimal control of state constrained integral equations} \label{ocp}
\subsection{Setting}

We consider an optimal control problem with running and initial-final state constraints, of the following type:
\begin{align}
  &(P)&& \displaystyle \min_{(u,y)\in \U \times \Y} \int _0^T \ell(u_t,y_t) \ud t +\phi(y_0,y_T) \label{opt}\\
  &\text{subject to }&& \displaystyle y_t=y_0+\int_0^t f(t,s,u_s,y_s)\ud s, \quad t \in [0,T], \label{dyn} \\
  &&& g(y_t)\le 0 ,\quad t \in [0,T], \label{stateconst} \\
  &&& \Phi^E(y_0,y_T) =0 ,\label{ifcons} \\
  &&& \Phi^I(y_0,y_T)  \le 0 , \label{ifconst}
\end{align}
where
\begin{equation*}
  \U :=  L^\infty ([0,T]; \R^m), \quad \Y := W^{1,\infty}([0,T];\R^n)
\end{equation*}
are the \emph{control space} and the \emph{state space}, respectively.

The data are $\ell \colon \R^m \times \R^n \rightarrow \R$,
$\phi \colon \R^n \times \R^n \rightarrow \R$,
$f \colon \R \times \R \times \R^m \times \R^n \rightarrow \R^n$,
$g \colon \R^n \rightarrow \R^r$,
$\Phi^E \colon \R^n \times \R^n \rightarrow \R^{s_E}$,
$\Phi^I \colon \R^n \times \R^n \rightarrow \R^{s_I}$ and $T>0$.
We set $\tau$ as the symbol for the first variable of $f$. Observe that if $D_\tau f = 0$, we recover an optimal control problem of a state constrained ODE.
We make the following assumption:
\begin{description}
 \item[(A0)] $\ell, \phi, f, g,\Phi^E, \Phi^I$ are of class $C^\infty$ and $f$ is Lipschitz.
\end{description}

We call \emph{trajectory} a pair $(u,y)\in \U \times \Y$ which satisfies the \emph{state equation}~\eqref{dyn}.
Under assumption \textbf{(A0)} it can be shown by standard contraction arguments that for any $(u,y_0) \in \U \times \R^n$,
the state equation~\eqref{dyn} has a unique solution $y$ in $\Y$, denoted by $y[u,y_0]$.
Moreover, the map $\Gamma \colon \U \times \R^n \rightarrow \Y$ defined by $\Gamma(u,y_0):=y[u,y_0]$ is of class $C^\infty$.

\subsection{Lagrange multipliers} \label{Lagrange multipliers}

The dual space of the space of vector-valued continuous functions $C\left( [0,T]; \R^r \right)$ is the space of finite vector-valued Radon measures
$\M\left([0,T]; \R^{r*} \right)$,
under the pairing
\begin{equation*}
 \left\langle \mu,h \right\rangle := \int_{[0,T]} \ud \mu_t h_t= \sum_{1\le i \le r} \int_{[0,T]}   h_{i,t}\ud \mu_{i,t.}
\end{equation*}
We define $BV\left([0,T];\R^{n*}\right)$, the space of vector-valued functions of bounded variations, as follows:
let $I$ be an open set which contains $[0,T]$; then
\begin{equation*} 
  BV\left([0,T];\R^{n*}\right) : = \left\{ h \in L^1(I;\R^{n*}) \, : \, Dh \in \M\left( I; \R^{n*}\right) , \, \supp(Dh) \subset [0,T] \right\} ,
\end{equation*}
 where $Dh$ is the distributional derivative of $h$; if $h$ is of bounded variations, we denote it by $\ud h$.
For $h \in  BV\left([0,T];\R^{n*}\right)$, there exists $h_{0_-},h_{T_+}\in \R^{n*}$ such that
\begin{equation} \label{moinsplus}
 \begin{array}{ll}
  h = h_{0_-} & \text{a.e. on } (- \infty,0) \cap I, \\
  h = h_{T_+} & \text{a.e. on } (T,+\infty) \cap I .
 \end{array}
\end{equation}
Conversely, we can identify any measure $\mu \in \M\left([0,T]; \R^{r*} \right)$ with the derivative of a function of bounded variations, denoted again by $\mu$,
such that $\mu_{T_+}=0$.
This motivates the notation $\ud \mu$ for any measure in the sequel, setting implicitly $\mu_{T_+}=0$.
See appendix~\ref{BVfunc} for more details.

\paragraph{}
Let
\begin{equation*}
 \M : = \M\left([0,T]; \R^{r*} \right) , \quad
 \cP : =  BV\left([0,T];\R^{n*}\right). 
\end{equation*}
We define the \emph{Hamiltonian} $H \colon \left[ \cP \right] \times \R \times \R^m \times \R^n \rightarrow \R$ by
\begin{equation} \label{Ham}
 H[p](t,u,y):= \ell(u,y)+p_t f(t,t,u,y) + \int_t^T p_s D_\tau f(s,t,u,y) \ud s 
\end{equation}
and the \emph{end points Lagrangian} $\Phi \colon \left[ \R^{s*} \right] \times \R^n \times \R^n \rightarrow \R$ by
\begin{equation} \label{endLag}
 \Phi[\Psi](y_1,y_2):= \phi(y_1,y_2)+\Psi \Phi(y_1,y_2)
\end{equation}
where $s:=s_E + s_I$ and $\Phi:=(\Phi^E,\Phi^I)$. We also denote $K:=\{ 0 \}_{s_E}\times \left(\R_-\right)^{s_I}$,
so that \eqref{ifcons}-\eqref{ifconst} can be rewritten as $\Phi(y_0,y_T) \in K$.
Given a trajectory $(u,y)$ and $(\ud \eta,\Psi) \in \M \times \R^{s*}$, the \emph{adjoint state} $p$,
whenever it exists, is defined as the solution in $\cP$ of
\begin{equation} \label{adjdyn}
 \left\{ \begin{array}{l}
          -\ud p_t = D_y H[p](t,u_t,y_t)\ud t +  \ud \eta_t g'(y_t),\\
          (-p_{0_-},p_{T_+}) = D\Phi[\Psi](y_0,y_T).
         \end{array} \right.
\end{equation}
Note that $\ud \eta_t g'(y_t) = \sum_{i=1}^r \ud \eta_{i,t} g_i'(y_t)$.
The adjoint state does not exist in general, but when it does it is unique.
More precisely, we have:

\begin{existadj} \label{existadj}
 There exists  a unique solution in $\cP$ of the adjoint state equation with final condition only (i.e. without initial condition):
 \begin{equation} \label{adjfinal}
 \left\{ \begin{array}{l}
          -\ud p_t = D_y H[p](t,u_t,y_t)\ud t + \ud \eta_t g'(y_t) ,\\
          \ \, p_{T_+} = D_{y_2}\Phi[\Psi](y_0,y_T)  .
         \end{array} \right.
\end{equation}
\end{existadj}

\begin{proof}
The contraction argument is given in appendix~\ref{BVfunc}.
\end{proof}

We can now define Lagrange multipliers for optimal control problems in our setting:
\begin{lagrangemult} \label{lagrangemult}
 $(\ud \eta, \Psi,p) \in \M \times \R^{s*} \times \cP$ is a \emph{Lagrange multiplier} associated with $(\bar u,\bar y)$ if
\begin{align}
 & p \text{ is the adjoint state associated with } (\bar u,\bar y,\ud \eta,\Psi) ,&& \label{multadj}\\
 & \ud \eta \ge 0 , \quad g(\bar y) \le 0, \quad \int_{[0,T]} \ud \eta_t g(\bar y_t)=0 , \label{multeta}\\
 & \Psi \in N_K\left( \Phi(\bar y_0,\bar y_T)\right) ,\label{multpsi}\\
 & D_u H[p](t,\bar u_t,\bar y_t)=0 \text{ for a.a. } t \in [0,T] .\label{multstation}
\end{align}
\end{lagrangemult}

\subsection{Linearized state equation} \label{linearized state equation}
For $s \in [1,\infty]$, let
\begin{equation*}
 \V_s :=  L^s([0,T];\R^m) , \quad
 \Z_s:= W^{1,s}([0,T];\R^n) .
\end{equation*}
Given a trajectory $(u,y)$ and $(v,z_0) \in \V_s \times \R^n$, we consider the \emph{linearized state equation} in $\Z_s$:
\begin{equation} \label{lindyn}
 z_t =  z_0 + \int_0^t D_{(u,y)} f(t,s,u_s,y_s)(v_s,z_s) \ud s.
\end{equation}
It is easily shown that there exists a unique solution $z \in \Z_s$ of~\eqref{lindyn},
called the \emph{linearized state} associated with the trajectory $(u,y)$ and the direction $(v,z_0)$,
and denoted by $z[v,z_0]$ (keeping in mind the nominal trajectory).

\begin{gron} \label{estimsobo}
There exists $C>0$ and $C_s > 0$ for any $s \in [1,\infty ]$
(depending on $(u,y)$) such that, for all $(v,z_0) \in \V_s \times \R^n$ and all $t \in [0,T]$,
 \begin{align}
& | z[v,z_0]_t | \le C  \left(  |z_0 |+ \int_0^t |v_s| \ud s \right), \label{ll1}\\
&  \| z[v,z_0] \|_{1,s} \le C_s  \left(|z_0 |+  \|v\|_s   \right) . \label{ll2}
 \end{align}
\end{gron}
\begin{proof}
 \eqref{ll1} is an application of Gronwall's lemma and \eqref{ll2} is a consequence of \eqref{ll1}.
\end{proof}
Observe that for $s=\infty$, the linearized state equation arises naturally:
let $(u,y_0) \in \U \times \R^n$, $y:=\Gamma(u,y_0) \in \Y$.
We consider the linearized state associated with the trajectory $(u,y)$ and a direction $(v,z_0) \in \U \times \R^n$. Then
\begin{equation}
 z[v,z_0 ] = D\Gamma(u,y_0)(v,z_0) .
\end{equation}
Similarly we can define the \emph{second-order linearized state}:
\begin{equation} \label{secondlin}
 z^2[v,z_0] :=  D^2\Gamma(u,y_0)(v,z_0)^2 .
\end{equation}
Note that $z^2[v,z_0]$ is the unique solution in $\Y$ of
\begin{equation}
 z^2_t = \int_0^t \left( D_y f(t,s,u_s,y_s)z^2_s + D^2_{(u,y)^2}f(t,s,u_s,y_s)(v_s,z[v,z_0]_s)^2 \right) \ud s .
\end{equation}

\subsection{Running state constraints} \label{running state constraints}
The running state constraints $g_i$, $i=1,\ldots,r$, are considered along trajectories $(u,y)$.
They produce functions of one variable, $t \mapsto g_i(y_t) $, which belong \textit{a priori} to $W^{1,\infty}([0,T])$ and satisfy
\begin{equation} \label{example}
 \frac{\ud}{\ud t}g_i(y_t) = g_i'(y_t) \left( f(t,t,u_t,y_t) + \int_0^t D_\tau f(t,s,u_s,y_s) \ud s \right).
\end{equation}
There are two parts in this derivative:
\begin{itemize}
 \itr $t \mapsto g_i'(y_t) f(t,t,u_t,y_t)$, where $u$ appears pointwisely.
 \itr $t \mapsto g_i'(y_t) \int_0^t D_\tau f(t,s,u_s,y_s) \ud s$, where $u$ appears in an integral.
\end{itemize}
Below we will distinguish these two behaviors and set $\tilde u$ as the symbol for the pointwise variable, $u$ for the integral variable
(similarly for $y$). If there is no dependance on $\tilde u$, one can differentiate again \eqref{example} w.r.t. $t$.
This motivates the definition of a notion of total derivative that always ``forget'' the dependence on $\tilde u$.
Let us do that formally.

\paragraph{}
First we need a set which is stable by operations such as in \eqref{example}, so that it will contain the derivatives of any order.
It is also of interest to know how the functions we consider depend on $(u,y) \in \U \times \Y$.
To answer this double issue, we define the following commutative ring:
\begin{equation}
\mathcal S:= \left\lbrace h \, :\, h(t,\tilde u,\tilde y,u,y)=    \label{s} 
\sum_\alpha  a_\alpha(t,\tilde u,\tilde y) \prod_\beta \int_0^t b_{\alpha,\beta}(t,s,u_s,y_s) \ud s  \right\rbrace , 
\end{equation}
where $(t,\tilde u,\tilde y,u,y) \in \R \times \R^m \times \R^n \times \U \times \Y$,
the $a_\alpha$, $b_{\alpha,\beta}$ are real functions of class $C^\infty$,
the sum and the products are finite and an empty product is equal to $1$.
The following is straightforward:
\begin{interest} \label{interest}
 Let $h \in \SS$, $(u,y) \in \U \times \Y$.
There exists $C>0$ such that, for a.a. $t \in [0,T]$ and for all $(\vt,\zt,v,z)\in \R^m \times \R^n \times \U \times \Y$,
\begin{equation}
 \left| D_{(\ut,\yt,u,y)} h(t,u_t,y_t,u,y) (\vt,\zt,v,z) \right| \le C 
 \left( |\vt| +|\zt| + \int_0^t \left( |v_s|+|z_s|\right) \ud s \right) .
\end{equation}

\end{interest}

Next we define the derivation $D^{(1)} \colon \mathcal S \longrightarrow \mathcal S$ as follows
(recall that we set $\tau$ as the symbol for the first variable of $f$ or $b$):
\begin{enumerate}
 \item for $h \colon (t,\tilde u,\tilde y) \in \R \times \R^m \times \R^n \mapsto a(t,\tilde u,\tilde y) \in \R$,
 \begin{multline*}
 \left( D^{(1)}h\right) (t,\tilde u,\tilde y,u,y) := D_t a(t,\tilde u,\tilde y)\\
 + D_{\tilde y}a(t,\tilde u,\tilde y)\left(f(t,t,\tilde u,\tilde y)+\int_0^t D_\tau f (t,s,u_s,y_s) \ud s \right).
 \end{multline*}
 
 \item for $h \colon (t,u,y) \in \R \times \U \times \Y \mapsto \int_0^t b(t,s,u_s,y_s) \ud s \in \R$,
 \begin{equation*} 
 \left( D^{(1)}h\right) (t,\tilde u,\tilde y,u,y):= b(t,t,\tilde u,\tilde y) + \int_0^t D_\tau b(t,s,u_s,y_s) \ud s.
 \end{equation*} 
 
 \item for any $h_1,h_2 \in \mathcal S$,
 \begin{gather*}
  \left( D^{(1)}(h_1+h_2)\right) =\left( D^{(1)}h_1 \right)+\left( D^{(1)}h_2\right) ,\\
  \left( D^{(1)}(h_1 h_2)\right) =\left( D^{(1)}h_1 \right)h_2+h_1\left( D^{(1)}h_2\right).
 \end{gather*}
\end{enumerate}
It is clear that $D^{(1)}h \in \SS$ for any $h \in \SS$.
The following formula, which is easily checked on $h=a(t,\tilde u,\tilde y)$ and $h=\int_0^t b(t,s,u_s,y_s) \ud s$, will be used for any $h \in \SS$:
\begin{multline} \label{total}
 \left(D^{(1)}h\right)(t,u_t,y_t,u,y) = D_t h(t,u_t,y_t,u,y)+D_{\tilde y}h(t,u_t,y_t,u,y)f(t,t,u_t,y_t) \\
 + D_{\tilde y}h(t,u_t,y_t,u,y) \int_0^t D_\tau f (t,s,u_s,y_s)\ud s .
\end{multline}

\paragraph{}
Let us now highlight two important properties of $D^{(1)}$. First, it is a notion of total derivative:
\begin{total} \label{lemma:total}
Let $h \in \mathcal S$ be such that $D_{\tilde u}h \equiv 0$, $(u,y) \in \U \times \Y$ be a trajectory and
\begin{equation}
  \p \colon t  \mapsto h(t,u_t,y_t,u,y).
\end{equation}
 Then $\varphi \in W^{1,\infty}([0,T])$
and
\begin{equation}
 \frac{\ud \varphi}{\ud t}(t)=\left( D^{(1)}h\right) (t,u_t,y_t,u,y).\label{form1}
\end{equation}
\end{total}

\begin{proof}
We write $h$ as in~\eqref{s}. If $D_{\tilde u} h \equiv 0$, then for any $u_0 \in \R^m$,
\begin{align}
 \p(t) & = h(t,u_0,y_t,u,y) \label{l1} \\
 & = \sum_\alpha  a_\alpha(t,u_0,y_t) \prod_\beta \int_0^t b_{\alpha,\beta}(t,s,u_s,y_s) \ud s .\label{l2}
\end{align}
By~\eqref{l2}, $\p \in W^{1,\infty}([0,T])$. And by~\eqref{l1},
\begin{align*}
 \frac{\ud \p}{\ud t}(t) & = D_t h(t,u_0,y_t,u,y)+D_{\tilde y} h(t,u_0,y_t,u,y) \dot y_t \\
 & = D_t h(t,u_t,y_t,u,y)+D_{\tilde y} h(t,u_t,y_t,u,y) \dot y_t
\end{align*}
since $D_{\tilde u} D_t h \equiv D_t D_{\tilde u} h \equiv 0$ and $D_{\tilde u} D_{\tilde y}h \equiv 0$.
Using the expression of $\dot y_t$ and \eqref{total}, we recognize \eqref{form1}.
\end{proof}

Second, it satisfies a principle of commutation with the linearization:
\begin{totalbis} \label{lemma:totalbis}
 Let $h, (u,y)$ be as in lemma~\ref{lemma:total}, $(v,z_0) \in \V_s \times \R^n$, $z:=z[v,z_0] \in \Z_s$ for some $s \in [1,\infty]$ and
\begin{equation}
 \psi \colon t  \mapsto D_{(\tilde y,u,y)}h (t,u_t,y_t,u,y)(z_t,v,z).
\end{equation}
Then  $\psi \in W^{1,s}([0,T])$ and
\begin{equation}
 \frac{\ud \psi}{\ud t}(t) = D_{(\tilde u,\tilde y,u,y)} \left[ \left(D^{(1)}h\right)(t,u_t, y_t,u,y)\right] (v_t,z_t,v,z) . \label{form2}
\end{equation}
\end{totalbis}

\begin{proof}
Using $D_{\tilde u} D_{(\tilde y,u,y)} h \equiv 0$, we have
\begin{align*}
 \psi(t) & = D_{(\tilde y,u,y)}h(t,u_0,y_t,u,y)(z_t,v,z) \\
 & =  \sum_\alpha  D_{\tilde y}a_\alpha(t,u_0,y_t)z_t \prod_\beta \int_0^t b_{\alpha,\beta} \ud s \\
 & \hspace{0.4cm}+ \sum_{\alpha,\beta}  a_\alpha(t,u_0,y_t) \int_0^t D_{(u,y)} b_{\alpha,\beta}(t,s,u_s,y_s)(v_s,z_s) \ud s 
  \prod_{\beta' \ne \beta} \int_0^t b_{\alpha,\beta'} \ud s. \nonumber
\end{align*}
It implies that $\psi \in W^{1,s}([0,T])$ and that
\begin{multline*}
 \frac{\ud \psi}{\ud t}(t)  = D^2_{t,(\tilde y,u,y)}h(t,u_t,y_t,u,y)(z_t,v,z) \\
  +D^2_{\tilde y,(\tilde y,u,y)}h(t,u_t,y_t,u,y)\left(\dot y_t,(z_t,v,z)\right) +D_{\tilde y}h(t,u_t,y_t,u,y) \dot z_t .
\end{multline*}
On the other hand, we differentiate $D^{(1)}h$ w.r.t. $(\tilde u,\tilde y,u,y)$ using \eqref{total}.
Then with the expressions of $\dot y_t$ and $\dot z_t$, we get the relation \eqref{form2}.
\end{proof}

\paragraph{}
Finally we define the order of a running state constraint $g_i$.
We denote $g_i^{(j+1)}:= D^{(1)}g_i^{(j)}$ (with $g_i^{(0)}:=g_i$).
Note that $g_i \in \mathcal S$, so $g_i^{(j)} \in \mathcal S$ for all $j\ge 0$.
Moreover, if we write $g_i^{(j)}$ as in~\eqref{s}, the $a_\alpha$ and $b_{\alpha,\beta}$ are combinations of derivatives of $f$ and $g_i$.

\begin{order} \label{order}
The \emph{order} of the constraint $g_i$ is the greatest positive integer $q_i$ such that
\begin{equation*}
 D_{\tilde u}g_i^{(j)} \equiv 0 \quad \text{for all } j=0,\ldots,q_i-1 .
\end{equation*}
\end{order}

\paragraph{}
We have a result similar to Lemma~9 in \cite{MR2421298}, but now for integral dynamics.
Let $(u,y) \in \U \times \Y $ be a trajectory, $(v,z_0) \in \V_s \times \R^n$, and $z:=z[v,z_0] \in \Z_s$ for some $s \in [1,\infty]$.
\begin{derivative} \label{derivative}
Let $g_i$ be of order at least $q_i \in \mathbb N^*$. Then
\[
 \begin{array}{ll}
  t \mapsto g_i(y_t) &  \in W^{q_i,\infty}([0,T]) ,\\
  t \mapsto g_i'(y_t)z_t & \in W^{q_i,s}([0,T]) ,
 \end{array}
\]
and
\begin{align}
&\frac{\ud^j}{\ud t^j}g_i(y)|_t =  g_i^{(j)}(t,y_t,u,y), \quad j=1,\ldots,q_i-1 \label{1q-1}, \\
&\frac{\ud^{q_i}}{\ud t^{q_i}}g_i(y)|_t  =  g_i^{(q_i)}(t,u_t,y_t,u,y)  \label{qs}, \\
&\frac{\ud^j}{\ud t^j}g_i'(y)z|_t  = \widehat D g_i^{(j)}(t,y_t,u,y)(z_t,v,z ), \quad j=1,\ldots,q_i-1 \label{ji},\\
&\frac{\ud^{q_i}}{\ud t^{q_i}}g_i'(y)z|_t  = D_{\tilde u}g_i^{(q_i)}(t,u_t,y_t,u,y)v_t + \widehat D g_i^{(q_i)}(t,u_t,y_t,u,y)(z_t,v,z ), \label{qu}
\end{align}
where we denote by $\D$ the differentiation w.r.t. $(\tilde y,u,y)$. 
\end{derivative}
\begin{proof}
 It is straightforward with lemmas~\ref{lemma:total} and~\ref{lemma:totalbis}, definition~\ref{order} and an induction on $j$.
\end{proof}

\section{Weak results} \label{weak}
\subsection{A first abstract formulation} \label{a first abstract formulation}

The optimal control problem ($P$) can be rewritten as an abstract optimization problem on $(u,y_0)$.
The most naive way to do that is the following equivalent formulation:
\begin{align}
  &(P)&& \displaystyle \min_{(u,y_0)\in \U \times \R^n} J(u,y_0) \label{abs1} \\
  &\text{subject to }&& g(y[u,y_0]) \in  C_- \left( [0,T]; \R^r \right) , \label{abs2} \\
  &&& \Phi(y_0,y[u,y_0]_T) \in  K , \label{abs3} 
\end{align}
where
\begin{equation} \label{cost}
  J(u,y_0):= \int _0^T \ell(u_t,y[u,y_0]_t) \ud t +\phi(y_0,y[u,y_0]_T) 
\end{equation}
and $\Phi = (\Phi^E, \Phi^I)$, $K=\{0\}_{s_E} \times \left(\R_-\right)^{s_I}$.
In order to write optimality conditions for this problem, we first compute its Lagrangian
\begin{equation}
  L (u,y_0, \ud \eta, \Psi):= J(u,y_0)  + \int_{[0,T]} \ud \eta_t g(y[u,y_0]_t) + \Psi \Phi(y_0,y[u,y_0]_T) 
\end{equation}
where $(u,y_0,\ud \eta,\Psi) \in \U \times \R^n \times \M \times \R^{s*}$ (see the beginning of section~\ref{Lagrange multipliers}).
A \emph{Lagrange multiplier} at $(u,y_0)$ in this setting is any $(\ud \eta,\Psi)$ such that
\begin{align}
& D_{(u,y_0)} L (u,y_0, \ud \eta, \Psi) \equiv 0, \label{abst1}\\
& (\ud \eta,\Psi) \in N_{C_- \left( [0,T]; \R^r \right) \times K} \left( g(y),\Phi(y_0,y_T) \right). \label{abst2}
\end{align}
This definition has to be compared to definition~\ref{lagrangemult}:
\begin{coupe} \label{coupe}
We have that
 $(\ud \eta,\Psi)$ is a Lagrange multiplier of the abstract problem \eqref{abs1}-\eqref{abs3} at $(\ub,\yb_0)$ \textnormal{iff}
$(\ud \eta,\Psi,p)$ is a Lagrange multiplier of the optimal control problem \eqref{opt}-\eqref{ifconst} associated with $(\ub,y[\ub,\yb_0])$,
where $p$ is the unique solution of \eqref{adjfinal}.
\end{coupe}

\begin{proof}
Using the Hamiltonian \eqref{Ham}, the end points Lagrangian \eqref{endLag} and the formula \eqref{ippform} of integration by parts for functions of bounded variations
(see appendix~\ref{BVfunc}), we get
\begin{multline*}
 L (u,y_0, \ud \eta, \Psi)= \int _0^T H[p](t,u_t,y_t) \ud t + \int_{[0,T]}  \left( \ud p_t y_t + \ud \eta_t g(y_t) \right)\\
  + p_{0_-}y_0 - p_{T_+}y_T  + \Phi[\Psi](y_0,y_T)  
\end{multline*}
for any $p\in \cP$ and $y=y[u,y_0]$. We fix $(\bar u,\bar y_0, \ud \eta, \Psi)$, we differentiate $L$ w.r.t. $(u, y_0)$ at this point,
and we choose $p$ as the unique solution of
\eqref{adjfinal}. Then
\begin{multline*}
 D_{(u,y_0)} L (\bar u,\bar y_0, \ud \eta, \Psi)(v,z_0)= \int _0^T D_u H[p](t,\bar u_t,\bar y_t)v_t \ud t  \\+ \left( p_{0_-} + D_{y_1} \Phi[\Psi](\bar y_0,\bar y_T)  \right)z_0
\end{multline*}
for all $(v,z_0) \in \U \times \R^n$. It follows that \eqref{abst1} is equivalent to \eqref{multadj} and \eqref{multstation}.
And it is obvious that \eqref{abst2} is equivalent to \eqref{multeta}-\eqref{multpsi}.
\end{proof}

Second we need a qualification condition.
%
\begin{qualif} \label{qualif}
We say that $(\bar u,\bar y)$ is \emph{qualified} if
\begin{enumerate}[(i)]
	    \item $\left\{ \ba{ccl}
		    (v,z_0)& \mapsto &D\Phi^E(\bar y_0,\bar y_T)(z_0,z[v,z_0]_T) \\
		    \U \times \R^n &\rightarrow &\R^{s_E}
		    \ea \right.$ is onto,
	    \item there exists $(\bar v,\bar z_0)  \in \U \times \R^n$ such that, with $\bar z= z[\bar v,\bar z_0]$,
	    \begin{equation*}
	      \left\{ \ba{ll}
	     D\Phi^E(\bar y_0,\bar y_T)(\bar z_0,\bar z_T)=0,& \\
	     D\Phi^I_i(\bar y_0,\bar y_T)(\bar z_0,\bar z_T) < 0 ,& i \in \left\{i\,:\,\Phi^I_i(\bar y_0,\bar y_T) = 0\right\}, \\
	     g_i'(\bar y_t)\bar z_t < 0 \ \text{on}\ \left\{t \,:\, g_i(\bar y_t) =0\right\},& i=1,\ldots,r .
		\ea \right.
	    \end{equation*}
\end{enumerate}
\end{qualif}

\begin{qualifrem}
 \begin{enumerate}
  \item This condition is equivalent to Robinson's constraint qualification (introduced in \cite{MR0410522}, Definition~2)
for the abstract problem \eqref{abs1}-\eqref{abs3} at $(\bar u,\bar y_0)$;
see the discussion that follows Definition~3.4 and Definition~3.5 in \cite{MR941318} for a proof of the equivalence.
 \item It is sometimes possible to give optimality conditions without qualification condition by considering an auxiliary optimization problem
(see e.g. the proof of Theorem~3.50 in \cite{MR1756264}).
Nevertheless, observe that if $(\ub,\yb)$ is feasible but not qualified because (i) does not hold, then there exists a \emph{singular Langrange multiplier}
of the form $(0,\Phi^E,0)$. One can see that second-order necessary conditions become pointless since $-(0,\Phi^E,0)$ is a singular Lagrange multiplier too.
\end{enumerate}

\end{qualifrem}

Finally we derive the following first-order necessary optimality conditions:
\begin{firsteasy} \label{firsteasy}
 Let $(\bar u,\bar y)$ be a qualified local solution of ($P$).
 Then the set of associated Lagrange multipliers is nonempty, convex, bounded and weakly~$*$ compact.
 \end{firsteasy}
 
\begin{proof}
 Since the abstract problem \eqref{abs1}-\eqref{abs3} is qualified, we get the result for the set $\{(\ud \eta,\Psi)\}$ of Lagrange multipliers in this setting
(Theorem~4.1 in \cite{MR526427}).
 We conclude with lemma~\ref{coupe} and the fact that
\begin{align*}
 \M \times \R^{s*} & \longrightarrow \M \times \R^{s*} \times \cP  \nonumber \\
(\ud \eta,\Psi)& \longmapsto (\ud \eta,\Psi,p)
\end{align*}
is affine continuous (it is obvious from the proof of lemma~\ref{existadj}).
\end{proof}

We will prove a stronger result in section~\ref{strong}, relying on another abstract formulation, the so-called \emph{reduced problem}.
The main motivation for the reduced problem, as mentioned in the introduction, is actually to satisfy an \emph{extended polyhedricity condition}
(see Definition~3.52 in \cite{MR1756264}), in order to easily get second-order necessary conditions (see Remark~3.47 in the same reference).

\subsection{The reduced problem} \label{reducedpbdef}
In the sequel \underline{we fix a feasible trajectory $(\bar u, \bar y)$},
i.e. which satisfies \eqref{dyn}-\eqref{ifconst},
and denote by $\Lambda$ the set of associated Lagrange multipliers (definition~\ref{lagrangemult}).
We need some definitions:

\begin{constraints}
 An \emph{arc} is a maximal interval, relatively open in $[0,T]$, denoted by $(\tau_1,\tau_2)$,
such that the set of active running state constraints at time $t$ is constant for all $t \in (\tau_1,\tau_2)$.
It includes intervals of the form $[0,\tau)$ or $(\tau,T]$.
 If $\tau$ does not belong to any arc, we say that $\tau$ is a \emph{junction time}.
 
 Consider an arc $(\tau_1,\tau_2)$. It is a \emph{boundary arc} for the constraint $g_i$ if the latter is active on $(\tau_1,\tau_2)$;
 otherwise it is an \emph{interior arc} for $g_i$.
 
 Consider an interior arc $(\tau_1,\tau_2)$ for $g_i$. If $g_i(\tau_2)=0$, then $\tau_2$ is an \emph{entry point} for $g_i$;
 if $g_i(\tau_1)=0$, then $\tau_1$ is an \emph{exit point} for $g_i$.
 If $\tau$ is an entry point and an exit point, then it is a \emph{touch point} for $g_i$. 
 
 Consider a touch point $\tau$ for $g_i$. We say that $\tau$ is \emph{reducible}
 if $\frac{\ud^2}{\ud t^2}g_i(\bar y_t)$, defined in a weak sense, is a function for $t$ close to $\tau$, continuous at $\tau$, and 
 \begin{equation*}
 \frac{\ud^2}{\ud t^2}g_i(\bar y_t) |_{t=\tau} < 0 .
 \end{equation*}

\end{constraints}

\begin{reducible}
 Let $\tau$ be a touch point for $g_i$. By lemma~\ref{derivative}, if $g_i$ is of order at least $2$, then
$\tau$ is reducible if $t \mapsto g_i^{(2)}(t,\bar u_t,\bar y_t,\bar u,\bar y)$
is continuous at $\tau$ and $g_i^{(2)}(\tau,\bar u_\tau,\bar y_\tau,\bar u,\bar y)< 0$.
Note that the continuity holds either if $u$ is continuous at $\tau$ or if $g_i$ is of order at least $3$.
\end{reducible}

\paragraph{}
The interest of reducibility will appear with the next lemma.
For $\tau \in [0,T]$ and $\e > 0$ (to be fixed), we define $\mu_{\tau} \colon W^{2,\infty}([0,T]) \rightarrow \R$ by
\begin{equation} \label{mudef}
  \mu_{\tau}(x) := \max\left\{ x_t \, : \, t \in [\tau- \e,\tau + \e] \cap [0,T] \right\}.
\end{equation}
\begin{muep}\label{muep} 
 Let $g_i$ not be of order $1$ (i.e. $ D_{\tilde u}g_i^{(1)} \equiv 0$) and $\tau$ be a reducible touch point for $g_i$.
 Then for $\e > 0$ small enough, $\mu_\tau$ is $C^1$ in a neighbourhood of $g_i(\yb) \in W^{2,\infty}([0,T])$
 and twice Fr\'echet differentiable at $g_i(\yb)$, with first and second derivatives at $g_i(\yb)$ given by
  \begin{align}
 & D\mu_{\tau}(g_i(\bar y))x =x_\tau , \label{mu1}\\
 & D^2\mu_{\tau}(g_i(\bar y))(x)^2 = - \dfrac{\left(\dfrac{\ud}{\ud t}x_t |_\tau \right)^2}{\dfrac{\ud^2}{\ud t^2}g_i(\yb_t)|_\tau},\label{mu2}
 \end{align}
for any $x \in W^{2,\infty}([0,T])$.
\end{muep}

\begin{proof}
We apply Lemma~23 of \cite{MR2421298} to $g_i(\yb)$, which belongs to $W^{2,\infty}([0,T])$ by lemma~\ref{derivative}
and satisfies the required hypotheses at $\tau$ by definition of a reducible touch point.
\end{proof}

\begin{muepsuite} \label{muepsuite}
We can write \eqref{mu1} and \eqref{mu2} for $x= g_i'(\yb)z[v,z_0]$ (since $g_i$ is not of order $1$).
By lemma~\ref{derivative}, \eqref{mu2} becomes
\begin{align}
& D^2\mu_{\tau}(g_i(\bar y))(g_i'(\yb)z)^2 = - \dfrac{\left( \D g_i^{(1)}(\tau,\bar y_\tau,\bar u,\bar y)(z_\tau,v,z)\right)^2}
{g_i^{(2)}(\tau,\bar u_\tau,\bar y_\tau,\bar u,\bar y)}  \label{mu4}
\end{align}
where $z=z[v,z_0]$, $(v,z_0) \in \U \times \R^n$ and $\D$ is the differentiation w.r.t. $(\tilde y,u,y)$.
We will also use \eqref{mu1} for $x= g_i''(\yb)(z[v,z_0])^2 + g_i'(\yb)z^2[v,z_0]$, $z^2[v,z_0]$ being defined by \eqref{secondlin}.
It can indeed be shown that it belongs to $W^{2,\infty}([0,T])$.
\end{muepsuite}

In view of these results we distinguish running state constraints of order $1$.
Without loss of generality, we suppose that 
\begin{itemize}
\itr $g_i$ is of order $1$ for $i= 1 ,\ldots, r_1$,
\itr $g_i$ is not of order $1$ for $i=r_1 +1 ,\ldots,r$,
\end{itemize}
where $0 \le r_1 \le r$.
We make now the following assumption:
\begin{description}
 \item[(A1)] There are finitely many junction times, and for $i=r_1+1 ,\ldots,r$ all touch points for $g_i$ are reducible.
\end{description}

\paragraph{}
For $i=1 , \ldots,r_1$ we consider the contact sets of the constraints
\begin{equation} \label{contact}
 \I_i  := \{t\in [0,T]\, : \, g_i(\bar y_t)=0 \} .
\end{equation}
For $i=r_1+1 , \ldots,r$ we remove the touch points from the contact sets:
 \begin{align}
 \mathcal T_i & := \text{the set of (reducible) touch points for } g_i, \label{touch}\\
 \I_i & := \{t\in [0,T]\, : \, g_i(\bar y_t)=0 \}\setminus \mathcal T_i.\label{bound}
 \end{align}
For $i=1 , \ldots,r$ and $\e \ge 0$ we denote
\begin{equation} \label{vois}
  \I_i^\e:= \{t\in [0,T]\, :\, \dist(t,\I_i) \le \e \} .
\end{equation}
Assumption \textbf{(A1)} implies that $\I_i^\e$ has finitely many connected components for any $\e \ge 0$ ($1\le i \le r$) and that $\T_i$ is finite ($1 \le i \le r_1 $).
Let $N:= \sum_{r_1 < i \le r} | \mathcal T_i |$.

\paragraph{}
Now \underline{we fix $\e > 0$} small enough (so that lemma~\ref{muep} holds) and we define
\begin{align}
 & G_1(u,y_0):= \big( g_i\left(y[u,y_0]\right)|_{\I_i^\e}\big)_{1\le i \le r} ,
 && K_1 := \prod_{i=1}^r C_{-}\left( \I_i^\e\right), \\
  & G_2(u,y_0):= \big( \mu_{\tau}\left(g_i\left(y[u,y_0]\right)\right) \big)_{\tau \in \mathcal T_i ,\  r_1< i \le r} ,
 && K_2 := \big( \R_{-}\big)^{N} , \\
 & G_3(u,y_0):=\Phi \left(y_0,y[u,y_0]_T\right), && K_3:= K .
\end{align}
Recall that $J$ has been defined by \eqref{cost}.

\paragraph{}
The \textbf{reduced problem} is the following abstract optimization problem:
\begin{align*}
  (P_R) \quad \displaystyle \min_{(u,y_0)\in \mathcal U \times \R^n} J(u,y_0), \quad \text{subject to} \quad 
  \left\{ \begin{array}{l}
   G_1(u,y_0) \in K_1 \\
   G_2(u,y_0) \in K_2 \\
   G_3(u,y_0) \in K_3  
  \end{array} \right. .
\end{align*}

\begin{link} \label{link}
We had fixed $(\ub,\yb)$ as a feasible trajectory; then $(\ub,\yb_0)$ is feasible for $(P_R)$.
Moreover, $(\bar u,\bar y)$ is a local solution of ($P$) \textit{iff}
$(\bar u,\bar y_0)$ is a local solution of ($P_R$),
%
and the qualification condition at $(\bar u,\bar y)$ (definition~\ref{qualif})
is equivalent to Robinson's constraints qualification for ($P_R$) at $(\bar u,\bar y_0)$
(using lemma~\ref{muep}).
\end{link}

Thus it is of interest for us to write optimality conditions for ($P_R$).

\subsection{Optimality conditions for the reduced problem}

The \emph{Lagrangian} of ($P_R$) is 
\begin{multline} \label{lagclass}
 L_R(u,y_0,\ud \rho,\nu,\Psi):= J(u,y_0) +\sum_{1 \le i \le r} \int_{\I_i^{\e}} g_i(y[u,y_0]_t)\ud \rho_{i,t}  \\
 + \sum_{\substack{\tau \in \mathcal T_i \\ r_1< i \le r}} \nu_{i,\tau}\mu_{\tau}\left(g_i\left(y[u,y_0]\right)\right) +\Psi \Phi(y_0,y[u,y_0]_T)
\end{multline}
\begin{flalign*}
&\text{where} \quad u \in \U, \quad y_0 \in \R^n, \quad \ud \rho \in \prod_{i=1}^r \M\left( \I_i^{\e}\right) , \quad \nu \in \R^{N*},\quad \Psi \in \R^{s*}.&&
\end{flalign*}
As before, a measure on a closed interval is denoted by $\ud \mu$
and is identified with the derivative of a function of bounded variations which is null on the right of the interval.

A \emph{Lagrange multiplier} of ($P_R$) at $(\bar u,\bar y_0)$ is any $(\ud \rho,\nu,\Psi)$ such that
\begin{align}
 & D_{(u,y_0)} L_R(\bar u,\bar y_0,\ud \rho,\nu,\Psi) = 0, \label{normalstation}\\
 & \ud \rho_i \ge 0 ,\ g_i(\bar y)|_{\I_i^\e} \le 0,\ \int_{\I_i^{\e}} g_i(\bar y_t) \ud \rho_{i,t} = 0, \quad i=1,\ldots,r,\label{normalrho} \\
 & \nu_{i,\tau} \ge 0, \ \mu_{\tau}\left(g_i(\yb )\right)\le 0,
 \ \nu_{i,\tau}\mu_{\tau}\left(g_i(\yb )\right)= 0,\quad \tau \in \mathcal T_i,\ i=r_1+1,\ldots,r, \label{normalnu} \\
 & \Psi \in N_K\left( \Phi(\bar y_0,\bar y_T) \right). \label{normalpsi}
\end{align}
We denote by $\Lambda_R$ the set of Lagrange multipliers of ($P_R$) at $(\bar u,\bar y_0)$.
The first-order necessary conditions for ($P_R$) are the same as in theorem~\ref{firsteasy}:

\begin{firstreduced} \label{firstreduced}
Let $(\bar u,\bar y_0)$ be a qualified local solution of ($P_R$).
Then $\Lambda _R$ is nonempty, convex, bounded and weakly~$*$ compact.
\end{firstreduced}

\paragraph{}
Given $(\ud \rho,\nu) \in \prod_{i=1}^r \M\left( \I_i^{\e}\right) \times \R^{N*}$, we define $\ud \eta \in \M$ by
\begin{equation}  \label{eta_def}
\ud \eta_{i} := \left\{ \begin{array}{lll}
                            \ud \rho_{i} & \text{on } \I_i^{\e}, &i=1,\ldots,r, \\
                            \sum_{\tau \in \mathcal T_i} \nu_{i,\tau} \delta_\tau & \text{elsewhere},& i=r_1+1,\ldots,r.
                           \end{array} \right.
\end{equation} 
Conversely, given $\ud \eta \in \M$, we define $(\ud \rho,\nu) \in \prod_{i=1}^r \M\left( \I_i^{\e}\right) \times \R^{N*}$ by
\begin{equation} \label{rhonu}
 \left\{ \ba{ll}
   \ud \rho_i  := {\ud \eta_i}|_{\I_i^{\e}}&  i=1,\ldots,r ,\\
   \nu_{i,\tau}  := \ud \eta_i( \{ \tau \}) &\tau \in \mathcal T_i , \quad i=r_1+1,\ldots,r .
\ea \right.
\end{equation}
In the sequel we use these definitions to identify $(\ud \rho,\nu)$ and $\ud \eta$, and we denote
\begin{equation}
 [\eta_{i,\tau}] := \ud \eta_i( \{ \tau \}).
\end{equation}
Recall that $\Lambda$ is the set of Lagrange multipliers associated with $(\bar u,\bar y)$ (definition~\ref{lagrangemult}).
We have a result similar to lemma~\ref{coupe}:

\begin{mult1} \label{mult1}
We have that $(\ud \rho,\nu,\Psi) \in \Lambda_R$ \textnormal{iff} $(\ud \eta,\Psi,p) \in \Lambda$,
with $p$ the unique solution of \eqref{adjfinal}.
\end{mult1}

\begin{proof}
With the identification between $(\ud \rho,\nu)$ and $\ud \eta$ given by \eqref{eta_def} and \eqref{rhonu},
it is clear that \eqref{normalrho}-\eqref{normalnu} are equivalent to \eqref{multeta}.
Let these relations be satisfied by $(\ud \rho,\nu,\Psi)$ and $(\ud \eta,\Psi)$. Then in particular
\begin{equation} \label{support}
\begin{array}{ll}
\supp(\ud \eta_i) =\supp(\ud \rho_i) \subset \I_i & i= 1,\ldots,r_1, \\
\supp(\ud \eta_i) =\supp(\ud \rho_i)\cup \supp(\sum \nu_{i,\tau} \delta_\tau) \subset \I_i \cup \mathcal T_i & i= r_1 + 1,\ldots,r.
\end{array}
\end{equation}
We claim that in this case \eqref{normalstation} is equivalent to \eqref{multadj} and \eqref{multstation}.
Indeed, using $H[p]$ defined by \eqref{Ham}, $\Phi[\Psi]$ by \eqref{endLag},
the integration by parts formula \eqref{ippform} and \eqref{support}, we have
\begin{multline} \label{formulatodiff}
 L_R (u,y_0, \ud \rho,\nu, \Psi)= \int _{[0,T]} \left( H[p](t,u_t,y_t) \ud t +  \ud p_t y_t \right) + p_{0_-}y_0 - p_{T_+}y_T   \\
 +\sum_{1 \le i \le r} \int_{\I_i}  g_i(y_t) \ud \eta_{i,t}
+ \sum_{\substack{\tau \in \mathcal T_i \\ r_1< i \le r}} [\eta_{i,\tau}] \mu_{\tau}\left(g_i(y )\right)  + \Phi[\Psi](y_0,y_T)
\end{multline}
for any $p \in \cP$ and $y=y[u,y_0]$.
Let us differentiate (say for $i>r_1$)
\begin{equation} \label{contrib}
\int_{\I_i}  g_i(y_t) \ud \eta_{i,t}
+ \sum_{\tau \in \mathcal T_i}[\eta_{i,\tau}] \mu_{\tau}\left(g_i(y )\right)
\end{equation}
w.r.t. $(u,y_0)$ at $(\bar u,\bar y_0)$ in the direction $(v,z_0)$ and use \eqref{mu1} and \eqref{support}; we get
\begin{equation*}
\int_{\I_i} g_i'(\bar y_t)z_t \ud \eta_{i,t} 
 + \sum_{\tau \in \mathcal T_i}[\eta_{i,\tau}]D\mu_{\tau}\left(g_i(\yb )\right)(g_i'(\yb)z)
 =  \int_{[0,T]} g_i'(\bar y_t)z_t \ud \eta_{i,t} 
\end{equation*}
where $z=z[v,z_0]$.
Let us now differentiate similarly the whole expression \eqref{formulatodiff} of $L_R$; we get
\begin{multline}\label{previous}
 \int_0^T D_u H[p](t,\bar u_t,\bar y_t)v_t \ud t 
+\int_{[0,T]} \big( D_y H[p](t,\bar u_t,\bar y_t)\ud t +\ud p_t +\ud \eta_t g'(\yb_t) \big)z_t \\
+\big( p_{0_-} + D_{y_1}\Phi[\Psi](\yb_0,\yb_T) \big)z_0
+\big( -p_{T_+} + D_{y_2}\Phi[\Psi](\yb_0,\yb_T) \big)z_T .
\end{multline}
Fixing $p$ as the unique solution of \eqref{adjfinal} in \eqref{previous} gives
\begin{multline*} 
 D_{(u,y_0)}L_R(\bar u,\bar y_0,\ud \rho,\nu,\Psi)(v,z_0) = \int_0^T D_u H[p](t,\bar u_t,\bar y_t)v_t \ud t\\ 
+ \big(p_{0_-} + D_{y_1}\Phi[\Psi](\bar y_0,\bar y_T) \big) z_0 .
\end{multline*}
It is now clear that \eqref{normalstation} is equivalent to \eqref{multadj} and \eqref{multstation}.
\end{proof}

\paragraph{}

For the second-order optimality conditions, we need to evaluate the Hessian of $L_R$.
For $\lambda =( \ud \eta,\Psi,p) \in \Lambda$,
$(v,z_0) \in \U \times \R^n$ and $z=z[v,z_0] \in \Y$, we denote
\begin{multline} \label{new}
 \mathcal J[\lambda](v,z_0)   :=  \int_0^T D^2_{(u,y)^2}H[p](t,\bar u_t,\bar y_t)(v_t,z_t)^2\ud t + D^2\Phi[\Psi](\bar y_0,\bar y_T)(z_0,z_T)^2 \\
 + \sum_{1 \le i \le r} \int_{\I_i} g_i''(\bar y_t)(z_t)^2  \ud \eta_{i,t} \\
+ \sum_{\substack{\tau \in \mathcal T_i \\ r_1< i \le r}}
[\eta_{i,\tau}]\left[ g_i''(\yb_\tau)(z_\tau)^2+D^2\mu_{\tau}\left(g_i(\yb )\right)(g_i'(\yb)z)^2 \right].
 \end{multline}
In view of \eqref{mu4} and \eqref{support}, we could also write
 \begin{multline} \label{old}
 \mathcal J[\lambda](v,z_0) =  \int_0^T D^2_{(u,y)^2}H[p](t,\bar u_t,\bar y_t)(v_t,z_t)^2 \ud t + D^2\Phi[\Psi](\bar y_0,\bar y_T)(z_0,z_T)^2 \\
 + \int_{[0,T]} \ud \eta_t g''(\bar y_t)(z_t)^2 
 - \sum_{\substack{\tau \in \mathcal T_i \\ r_1< i \le r}}
 [\eta_{i,\tau}]\frac{\left( \D g_i^{(1)}(\tau,\bar y_\tau,\bar u,\bar y)(z_\tau,v,z)\right)^2}
 {g_i^{(2)}(\tau,\bar u_\tau,\bar y_\tau,\bar u,\bar y)} .
 \end{multline}

\begin{mult2} \label{mult2}
 Let $(\ud \rho,\nu,\Psi) \in \Lambda_R$. Let $\lambda=(\ud \eta,\Psi,p) \in \Lambda$ be as in lemma~\ref{mult1}.
Then for all $(v,z_0) \in \U \times \R^n$,
\begin{equation}
   D^2_{(u,y_0)^2}L_R(\bar u,\bar y_0,\ud \rho,\nu,\Psi)(v,z_0)^2  = \mathcal J[\lambda](v,z_0).
\end{equation}

\end{mult2}

\begin{proof}
We will use \eqref{formulatodiff} and \eqref{contrib} from the previous proof.
First we differentiate \eqref{contrib} twice w.r.t. $(u,y_0)$ at $(\bar u,\bar y_0)$ in the direction $(v,z_0)$.
Denoting $z=z[v,z_0]$ and $z^2=z^2[v,z_0]$, defined by \eqref{secondlin}, we get
\begin{multline*}
\int_{\I_i}\left( g_i''(\bar y_t)(z_t)^2 + g_i'(\bar y_t)z^2_t \right) \ud \eta_{i,t} \\
 +\sum_{\tau \in \mathcal T_i}[\eta_{i,\tau}] \left[ 
 D^2\mu_{\tau}\left(g_i(\yb )\right)(g_i'(\yb)z)^2 
 +   D\mu_{\tau}\left(g_i(\yb )\right) \big( g_i''(\yb)(z)^2 + g_i'(\yb)z^2 \big)  \right]\\
 = \int_{\I_i} g_i''(\bar y_t)(z_t)^2  \ud \eta_{i,t} +  \int_{[0,T]} g_i'(\bar y_t)z^2_t \ud \eta_{i,t} \\
+\sum_{\tau \in \mathcal T_i} [\eta_{i,\tau}]\left[D^2\mu_{\tau}\left(g_i(\yb )\right)(g_i'(\yb)z)^2 +g_i''(\yb_\tau)(z_\tau)^2 \right] ,
\end{multline*}
where we have used remark~\ref{muepsuite}, \eqref{mu1} and \eqref{support}.
Second we differentiate $L_R$ twice using \eqref{formulatodiff}
and then we fix $p$ as the unique solution of \eqref{adjfinal}.
The result follows as in the proof of lemma~\ref{mult1}.
\end{proof}

\paragraph{}
Suppose that $\Lambda \ne \emptyset$  
and let $\bar \lambda=(\ud \bar \eta,\bar \Psi,\bar p) \in  \Lambda$.
We define the \emph{critical $L^2$ cone} as the set $C_2$ of $(v,z_0)\in \V_2 \times  \R^n$ such that
\begin{align}
 & \left\{ \begin{array}{ll}
             g_i'(\bar y)z \le 0 &  \text{on } \I_i,\\
             g_i'(\bar y)z =0 & \text{on } \supp\left(\ud \bar \eta_i \right)\cap \I_i, 
           \end{array} \right.
\ i=1,\ldots,r, \label{critical1}\\
&  \left\{ \begin{array}{l}
            g_i'(\bar y_\tau)z_\tau \le 0, \\ 
	    \left[\bar \eta_{i,\tau} \right] g_i'(\bar y_\tau)z_\tau= 0,
           \end{array}\right.
\ \tau \in \mathcal T_i, \ i=r_1+1, \ldots,r , \label{critical2}\\
 & \left\{ \begin{array}{l}
            D\Phi(\bar y_0,\bar y_T)(z_0,z_T) \in T_K \left(\Phi(\bar y_0,\bar y_T)\right), \\
            \bar \Psi D\Phi(\bar y_0,\bar y_T)(z_0,z_T)=0,
           \end{array} \right. \label{critical3}
\end{align}
where $z=z[v,z_0] \in \Z_2$.
Then the \emph{critical cone} for ($P_R$) 
(see Proposition~3.10 in~\cite{MR1756264}) is the set
\begin{equation*}
 C_\infty := C_2 \cap \left(\U \times \R^n \right),
\end{equation*}
and the \emph{cone of radial critical directions} for ($P_R$)
(see Definition~3.52 in \cite{MR1756264}) is the set
\begin{equation*}
  C^R_\infty := \left\{ (v,z_0)\in C_\infty \, : \, \exists \bar \sigma > 0 \, : \, 
g_i(\bar y) + \bar \sigma g_i'(\bar y)z \le 0 \ \text{on } \I_i^{\e}, \ i=1,\ldots,r \right\},
\end{equation*}
where $z=z[v,z_0] \in \Y$.
These three cones do not depend on the choice of $\bar \lambda$.
In view of lemma~\ref{mult2}, the second-order necessary conditions for ($P_R$) can be written as follows:
\begin{absecond} \label{absecond}
Let $(\bar u,\bar y_0)$ be a qualified local solution of ($P_R$).
Then for any $(v,z_0) \in C^R_\infty$, there exists $\lambda \in \Lambda$ such that
\begin{equation}
  \mathcal J[\lambda](v,z_0) \ge 0 .
\end{equation}
\end{absecond}
\begin{proof}
Corollary~5.1 in \cite{MR941318}.
\end{proof}

\section{Strong results} \label{strong}

Recall that $(\ub,\yb)$ is a feasible trajectory that has been fixed to define the reduced problem at the beginning of section~\ref{reducedpbdef}.

\subsection{Extra assumptions and consequences}

We were so far under the assumptions \textbf{(A0)}-\textbf{(A1)}.
We make now some extra assumptions, which will imply a partial qualification of the running state constraints,
as well as the density of $C^R_\infty$ in a larger critical cone.
\begin{description}
 \item[(A2)] Each running state constraint $g_i, i=1 , \ldots,r$ is of finite order $q_i$.
 \item[Notations]
Given a subset $J\subset \{1,\ldots,r\}$, say $J=\{i_1<\cdots<i_l\}$, we define
$G^{(q)}_J \colon \R \times \R^m \times \R^n \times \U \times \Y \rightarrow \R^{|J|}$ by
\begin{equation} \label{nota}
 G^{(q)}_J(t,\tilde u,\tilde y,u,y):=
\begin{pmatrix}
 \bar g_{i_1}^{(q_{i_1})}(t,\tilde u,\tilde y,u,y) \\
 \vdots\\
 \bar g_{i_l}^{(q_{i_l})}(t,\tilde u,\tilde y,u,y)
\end{pmatrix}.
\end{equation}
For $\e_0 \ge 0$ and $t \in [0,T]$, let
\begin{align}
 I^{\e_0}_t & := \left\{ 1 \le i \le r \, : \, t \in \I^{\e_0}_i \right\}, \\
 M_t^{\e_0} & := D_{\tilde u}G^{(q)}_{I^{\e_0}_t}(t,\bar u_t,\bar y_t,\bar u,\bar y) \in \R^{|I^{\e_0}_t|} \times \R^{m*}.
\end{align}
 \item[(A3)] There exists $\e_0, \gamma >0$ such that, for all $t\in [0,T]$,
 \begin{equation}
\left|  (M_t^{\e_0})^T \xi \right| \ge \gamma \left|\xi\right| 
\quad \forall \xi \in \R^{|I^{\e_0}_t|} . \label{linindep}
 \end{equation} 
 \item[(A4)] The initial condition satisfies $g(\bar y_0) <0 $ and the final time $T$ is not an entry point
(i.e. there exists $\tau <T$ such that the set $I_t^0$ of active constraints at time $t$ is constant for $t \in (\tau, T]$).
\end{description}

\begin{assump} \label{assump}
 \begin{enumerate}
 \item We do not assume that $\ub$ is continuous, as was done in \cite{MR2779110}.
 \item Recall that $\e$ has been fixed to define the reduced problem. Without loss of generality we suppose that $\e_0 > \e$,
$\e_0 < \min\{ \tau \, : \, \tau \text{ junction times} \}$ and $2\e_0< \min \{ |\tau-\tau'| \, : \, \tau,\tau' \text{ distinct junction times} \}$.
We omit it in the notation $M_t^{\e_0}$.
 \item In some cases, we can treat the case where $ T $ is an entry point, say for the constraint $g_i$:
\begin{itemize}
 \itr if $1\le i \le r_1$ (i.e. if $q_i=1$), then what follows works similarly.
 \itr if $r_1 < i \le r$ (i.e. if $q_i >1$) and $\frac{\ud}{\ud t}g_i(\bar y_t) |_{t=T} > 0$,
then we can replace in the reduced problem $g_i(y[u,y_0])|_{[T-\e,T]} \le 0$ by the final constraint $g_i(y[u,y_0]_T) \le 0$.
\end{itemize}
\item By \textbf{(A1)}, we can write 
\begin{equation} \label{decompoT}
 [0,T] =  J_0 \cup \cdots \cup J_\kappa
\end{equation}
where $J_l$ ($l=0,\ldots, \kappa$) are the maximal intervals in $[0,T]$
such that $I^{\e_0}_t$ is constant (say equal to $I_l$) for $t \in J_l$.
We order $J_0,\ldots,J_\kappa$ in $[0,T]$.
Observe that for any $l \ge 1$, $ \overline{J_{l-1}} \cap \overline{J_l} = \{\tau \pm \e_0 \}$ with $\tau$ a junction time.
\end{enumerate}
\end{assump}

\paragraph{}
For $s \in [1, \infty]$, we denote
\begin{equation}
  W^{(q),s}([0,T]):= \prod_{i=1}^r W^{q_i,s}([0,T]), \quad  W^{(q),s}(\I^\e):=\prod_{i=1}^r W^{q_i,s}(\I_i^\e),
\end{equation}
and for $\varphi = \begin{pmatrix}
\varphi_1\\
\vdots\\
\varphi_r
\end{pmatrix} \in  W^{(q),s}([0,T])$,
$\varphi|_{\I^\e} := \begin{pmatrix}
\varphi_1|_{\I_1^\e}\\
\vdots\\
\varphi_r|_{\I_r^\e}
\end{pmatrix} \in W^{(q),s}(\I^\e)$.

Using lemma~\ref{derivative} we define, for $s \in [1, \infty]$ and $z_0 \in \R^n$,
 \begin{align}
\A_{s,z_0} \colon \V_s & \longrightarrow  W^{(q),s}([0,T]) \nonumber \\
v& \longmapsto g'(\bar y)z[v,z_0]  . 
 \end{align}

\paragraph{}
We give now the statement of a lemma in two parts, which will be of great interest for us
(particularly in section~\ref{densityresult}).
The proof is technical and can be skipped at a first reading. It is given in the next section.

\begin{bigresult} \label{bigresult}

\begin{enumerate}[a)]
 \item  Let $s \in [1, \infty]$ and $z_0 \in \R^n$.
Let $\bb \in W^{(q),s}(\I^\e)$.
Then there exists $v \in \V_s$ such that
\begin{equation} 
 \left( \A_{s,z_0}v\right) |_{\I^\e} =  \bb  .
\end{equation}
\item Let $z_0 \in \R^n$.
Let $(\bb,\vb) \in W^{(q),2}(\I^\e)\times \V_2$ be such that
\begin{equation} \label{hypb}
\left( \A_{2,z_0}\vb \right) |_{\I^\e} =  \bb.
\end{equation}
Let $b^k \in W^{(q),\infty}(\I^\e)$, $k\in \NN$, be such that $b^k \xrightarrow[]{W^{(q),2}(\I^\e)} \bb$.
Then there exists $v^k \in \U$, $k \in \NN$, such that $v^k\xrightarrow[]{L^2}\vb$ and
\begin{equation}
 \left( \A_{\infty,z_0}v^k\right) |_{\I^\e} =  b^k.
\end{equation}
\end{enumerate}

\end{bigresult}

\subsection{A technical proof}
In this section we prove lemma~\ref{bigresult}.
The proofs of a) and b) are very similar; in both cases we proceed in $\kappa+1$ steps using the decomposition
\eqref{decompoT} of $[0,T]$.
At each step, we will use the following two lemmas, proved in appendixes~\ref{noass} and \ref{ass3}, respectively.

The first one uses only \textbf{(A1)} and the definitions that follow.

\begin{raccords} \label{raccords}
Let $t_0 : = \tau \pm \e_0$ where $\tau$ is a junction time.
\begin{enumerate}[a)]
 \item  Let $s \in [1, \infty]$ and $z_0 \in \R^n$.
Let $(\bb,v) \in W^{(q),s}(\I^\e)\times \V_s$ be such that
\begin{equation}
\left( \A_{s,z_0}  v \right) |_{\I^\e } =  \bb \ \text{on } [0,t_0].
\end{equation}
Then we can extend $\bb$
to $\bt \in W^{(q),s}([0,T])$
in such a way that
\begin{equation} \label{hypa1}
 \bt =  \A_{s,z_0} v \ \text{on } [0,t_0] . 
\end{equation}
\item Let $z_0 \in \R^n$.
Let $(\bb,\vb) \in W^{(q),2}(\I^\e)\times \V_2$ be such that
\begin{equation}
 \left( \A_{2,z_0} \vb \right) |_{\I^\e} =  \bb .
\end{equation}
Let $(b^k,v^k)\in W^{(q),\infty}(\I^\e)\times \U $, $k \in \NN$, be such that
$(b^k,v^k)\xrightarrow[]{W^{(q),2}\times L^2} (\bb,\vb)$ and
\begin{equation}
\left( \A_{\infty,z_0}v^k \right) |_{\I^\e } =  b^k\ \text{on } [0,t_0].
\end{equation}
Then we can extend $b^k$ to $\bt^k \in W^{(q),\infty}([0,T])$, $k \in \NN$, in such a way that
$\bt^k \xrightarrow[]{W^{(q),2}([0,T])} \A_{2,z_0} \vb$ and
\begin{equation}
 \bt^k = \A_{\infty,z_0} v^k \ \text{on } [0,t_0] .
\end{equation}
\end{enumerate}
\end{raccords}

The second lemma relies on \textbf{(A3)}.

\begin{resol} \label{resol} 
Let $s \in [1, \infty]$ and $z_0 \in \R^n$.
 Let $l$ be such that $I_l \ne \emptyset$. For $t \in J_l$, we denote
(recall that $\D$ is the differentiation w.r.t. $(\yt,u,y)$)
\begin{equation}
\left\{ \begin{array}{l}
 M_t := D_{\tilde u}G^{(q)}_{I_l}(t,\bar u_t,\bar y_t,\bar u,\bar y) \in \R^{|I_l|} \times \R^{m*},\\
 N_t := \D G^{(q)}_{I_l}(t,\bar u_t,\bar y_t,\bar u,\bar y) \in \R^{|I_l|} \times \R^{n*}\times \U^* \times\Y^* .
\end{array} \right.
\end{equation}
\begin{enumerate}[a)]
 \item  Let $(\hb,v) \in L^s(J_l;\R^{|I_l|}) \times \V_s$. Then there exists $\vt \in \V_s$ such that
\begin{equation} \label{efb}
 \left\{ \ba{l}
\vt  = v \  \text{on } J_0 \cup \cdots \cup J_{l-1},\\
M_t \vt_t + N_t \left( z[\vt,z_0]_t,\vt,z[\vt,z_0]\right) = \hb_t \ \text{for a.a. } t \in J_l .
\ea \right.
\end{equation}
\item Let $(\hb,\vb) \in L^s(J_l;\R^{|I_l|}) \times \V_s$ be such that
\begin{equation}
M_t \vb_t + N_t \left( z[\vb,z_0]_t,\vb,z[\vb,z_0]\right) = \hb_t \ \text{for a.a. } t \in J_l .
\end{equation}
Let $(h^k,v^k) \in L^\infty(J_l; \R^{|I_l|}) \times \U$, $k \in \NN$, be such that
$(h^k,v^k)\xrightarrow[]{L^s\times L^s} (\hb,\vb)$.
Then there exists $\vt^k \in \U$, $k \in \NN$, such that
$\vt ^k \xrightarrow[]{L^s} \vb$ and
\begin{equation} \label{hfb}
 \left\{ \ba{l}
\vt^k = v^k  \  \text{on } J_0 \cup \cdots \cup J_{l-1},\\
M_t \vt^k_t + N_t \left( z[\vt^k,z_0]_t,\vt^k,z[\vt^k,z_0]\right) = h^k_t \ \text{for a.a. } t \in J_l .
\ea \right.
\end{equation}
\end{enumerate}
\end{resol}

\paragraph{}
\begin{proof}[Proof of lemma~\ref{bigresult}]
In the sequel we omit $z_0$ in the notations.

\bigskip
\noindent a)
Let $\bb \in W^{(q),s}(\I^\e)$.
We need to find $v \in \V_s$ such that
\begin{equation} \label{but1}
  g_i'(\yb)z[v] =  \bb_i \ \text{on } \I_i^\e , \ i=1,\ldots,r .
\end{equation}
Since
\begin{equation*}
 v=v' \ \text{on } [0,t] \Longrightarrow z[v]=z[v'] \ \text{on } [0,t] ,
\end{equation*}
let us construct $v^0,\ldots,v^\kappa \in \V_s$ such that, for all $l$,
\begin{equation*} 
 \left\{ \ba{l}
v^l = v^{l-1} \ \text{on } J_0 \cup \cdots \cup J_{l-1}, \\
g_i'(\yb)z[v^l] =  \bb_i \ \text{on } \I_i^\e \cap J_l , \ i=1,\ldots,r 
\ea \right.
\end{equation*}
and $v := v^\kappa$ will satisfy \eqref{but1}.

By \textbf{(A4)}, $J_0 = [0,\tau_1 - \e_0)$ where $\tau_1$ is the first junction time
and $ \I_i^\e \cap J_0 = \emptyset$ for all $i$.
Then we can choose $v^0 :=0$.

Suppose we have $v^0,\ldots v^{l-1}$ for some $l \ge 1$ and let us construct $v^l$.
Applying lemma~\ref{raccords}~a) to $(\bb,v^{l-1})$ with $\{t_0\} = \overline{J_{l-1}} \cap \overline{J_l}$, we get $\bt \in W^{(q),s}([0,T])$.
Since $\I_i^\e \cap J_l = \emptyset$ if $i \not \in I_l$, it is now enough to find $v^l$ such that
\begin{equation} \label{but4}
 \left\{ \ba{l}
v^l = v^{l-1} \ \text{on } J_0 \cup \cdots \cup J_{l-1}, \\
g_i'(\yb)z[v^l] = \bt_i\ \text{on } J_l , \ i \in I_l.
\ea \right.
\end{equation}
Suppose that $v^l = v^{l-1}$ on $J_0 \cup \cdots \cup J_{l-1}$.
Then $g_i'(\yb)z[v^l]= \bt_i$ on $J_{l-1}$,
and it follows
that
\begin{gather}
  g_i'(\yb)z[v^l] =  \bt_i \ \text{on } J_l \\
\Updownarrow\nonumber \\
 \displaystyle \frac{d^{q_i}}{dt^{q_i}} g_i'(\yb)z[v^l] =\frac{d^{q_i}}{dt^{q_i}} \bt_i \ \text{on }  J_l .\label{enfin}
\end{gather}
And by lemma~\ref{derivative}, \eqref{enfin} is equivalent to
\begin{equation}
 D_{\ut}g_i^{(q_i)}(t,\ub_t,\yb_t,\ub,\yb)v^l_t
+ \D g_i^{(q_i)}(t,\ub_t,\yb_t,\ub,\yb)(z[v^l]_t,v^l,z[v^l]) = \bt_i^{(q_i)}(t)
\end{equation}
for a.a. $t \in J_l$.

If $I_l = \emptyset$, we choose $v^l:=v^{l-1}$.
Otherwise, say $I_l=\{ i_1 < \cdots < i_p \}$ and define on $J_l$
\begin{equation*} 
\hb:=
\begin{pmatrix}
 \bt_{i_1}^{(q_{i_1})} \\
\vdots\\
 \bt_{i_p}^{(q_{i_p})}
\end{pmatrix}
\in L^s(J_l;\R^{|I_l|}) .
\end{equation*}
Then \eqref{but4} is equivalent to
\begin{equation} \label{but5}
 \left\{ \ba{l}
v^l = v^{l-1} \ \text{on } J_0 \cup \cdots \cup J_{l-1}, \\
M_t v^l_t + N_t (z[v^l]_t,v^l,z[v^l]) =\hb_t \ \text {for a.a. } t \in J_l.
\ea \right.
\end{equation}
Applying lemma~\ref{resol}~a) to $(h,v^{l-1})$, we get $\vt$ such that \eqref{but5} holds; we choose $v^l:= \vt$.

\bigskip
\noindent b)
We follow a similar scheme to the one of the proof of a).

Let $(\bb,\vb) \in W^{(q),2}(\I^\e)\times \V_2$ be such that
\begin{equation*}
 g_i'(\yb)z[\vb] =  \bb_i \ \text{on } \I^\e, \ i=1,\ldots,r.
\end{equation*}
Let $b^k \in W^{(q),\infty}(\I^\e)$, $k\in \NN$, be such that $b^k \xrightarrow[]{W^{(q),2}} \bb$.
Let us construct $v^{k,0},\ldots,v^{k,\kappa} \in \U$, $k \in \NN$, such that for all $l$,
$v^{k,l} \xrightarrow[k \to \infty]{L^2} \vb $ and
\begin{equation*}
 \left\{ \ba{l}
v^{k,l} = v^{k,l-1} \ \text{on } J_0 \cup \cdots \cup J_{l-1},  \\
g_i'(\yb)z[v^{k,l}] =  b_i^k \ \text{on } \I_i^\e \cap J_l , \ i \in I_l .
\ea \right.
\end{equation*}
We will conclude the proof by defining $v^k := v^{k,\kappa}$, $k \in \NN$.

We choose for $v^{k,0}$ the truncation of $\vb$, $k \in \NN$ (see definition~\ref{truncation} in appendix~\ref{ass3}).

Suppose we have $v^{k,0},\ldots,v^{k,l-1}$, $k \in \NN$, for some $l \ge 1$ and let us construct $v^{k,l}$, $k \in \NN$.
Applying lemma~\ref{raccords}~b) to $(b^k,v^{k,l-1})$ with $\{t_0 \} = \overline{J_{l-1}} \cap \overline{J_l}$,
we get $\bt^k \in W^{(q),\infty}([0,T])$, $k \in \NN$. In particular,
\begin{equation} \label{limit1}
 \bt^k \xrightarrow[]{W^{(q),2}} \bt
\end{equation}
where $ \bt : = g'(\yb)z[\vb] \in W^{(q),2}([0,T])$.
And it is now enough to find $v^{k,l}$, $k \in \NN$, such that
$v^{k,l} \xrightarrow[k \to \infty]{L^2} \vb $ and
\begin{equation} \label{goal}
  \left\{ \ba{l}
v^{k,l} = v^{k,l-1} \ \text{on } J_0 \cup \cdots \cup J_{l-1},  \\
g_i'(\yb)z[v^{k,l}] =  \bt_i^k \ \text{on }  J_l , \ i \in I_l.
\ea \right.
\end{equation}

If $I_l = \emptyset$, we choose $v^{k,l}=v^{k,l-1}$, $k \in \NN$.
Otherwise,  say $I_l=\{ i_1 < \cdots < i_p \}$ and define on $J_l$
\begin{equation*} 
\hb :=
\begin{pmatrix}
 \bt_{i_1}^{(q_{i_1})} \\
\vdots\\
 \bt_{i_p}^{(q_{i_p})}
\end{pmatrix}
\in L^2(J_l;\R^{|I_l|}), \quad 
h^k :=
\begin{pmatrix}
 (\bt_{i_1}^k)^{(q_{i_1})} \\
\vdots\\
 (\bt_{i_p}^k)^{(q_{i_p})}
\end{pmatrix}
\in L^\infty(J_l;\R^{|I_l|}).
\end{equation*}
We have
\begin{equation*} 
 M_t \vb_t + N_t(z[\vb]_t,\vb,z[\vb]) = \hb_t  \ \text{for a.a } t \in J_l 
\end{equation*}
and \eqref{goal} is equivalent to
\begin{equation} \label{goal2}
 \left\{ \ba{l}
v^{k,l} = v^{k,l-1} \ \text{on } J_0 \cup \cdots \cup J_{l-1},  \\
M_t v^{k,l}_t + N_t (z[v^{k,l}]_t,v^{k,l},z[v^{k,l}]) = h^k_t \ \text {for a.a. } t \in J_l.
\ea \right.
\end{equation}
By \eqref{limit1}, $h^k \xrightarrow[]{L^2} \hb$, and by assumption, $v^{k,l-1} \xrightarrow[k \to \infty]{L^2} \vb $.
Applying lemma~\ref{resol}~b) to $(h^k,v^{k,l-1})$, we get $\vt^k$, $k \in \NN$, such that 
$\vt^k \xrightarrow[]{L^2} \vb $ and \eqref{goal2} holds; we choose $v^{k,l} =\vt^k$, $k \in \NN$.

\end{proof}

\subsection{Necessary conditions} \label{noc}
Recall that we are under the assumptions \textbf{(A0)}-\textbf{(A4)}.

\subsubsection{Structure of the set of Lagrange multipliers}
Recall that we denote by $\Lambda$ the set of Lagrange multipliers associated with $(\bar u,\bar y)$ (definition~\ref{lagrangemult}).
We consider the projection map 
\begin{equation*}
   \begin{array}{rccc}
\pi :& \M \times \R^{s*} \times \cP  & \longrightarrow & \R^{N*} \times \R^{s*} \\
 &(\ud \eta ,\Psi,p) & \longmapsto & \left( \left( [\eta_{i,\tau}]\right)_{\tau,i} , \Psi \right)
 \end{array}
\end{equation*}
where $\tau \in \mathcal T_i$, $i=r_1+1, \ldots,r$.
A consequence of lemma~\ref{bigresult}~a) is the following:

\begin{inj} \label{inj}
 $\pi|_{\Lambda}$ is injective.
\end{inj}

\begin{proof}
%
We will use the fact that one of the constraint, namely $G_1$, has a surjective derivative.
For $ \ud \rho \in \prod_{i=1}^r \M\left( \I_i^{\e}\right)$, we define $F_\rho \in \left( W^{(q),\infty}(\I^\e) \right)^*$ by
\begin{equation*}
  F_\rho(\p):= \sum_{1\le i\le r} \int_{\I_i^{\e}} \p_{i,t} \ud \rho_{i,t}  \quad \text{for all } \p \in  W^{(q),\infty}(\I^\e).
\end{equation*}
Since by lemma~\ref{derivative},
$DG_1(\bar u,\bar y_0)(v,z_0) \in  W^{(q),\infty}(\I^\e)$ for all $(v,z_0) \in \U \times \R^n$, we have
\begin{align*}
  \left\langle \ud \rho, DG_1(\bar u,\bar y_0)(v,z_0)  \right\rangle &=  \left\langle F_\rho, DG_1(\bar u,\bar y_0)(v,z_0)  \right\rangle  \\
  & = \left\langle \left(DG_1(\bar u,\bar y_0)\right)^*F_\rho, (v,z_0)  \right\rangle .
\end{align*}
Then differentiating $L_R$, defined by \eqref{lagclass}, w.r.t. $(u, y_0)$ we get
\begin{multline} \label{station}
D_{(u,y_0)} L_R(\bar u,\bar y_0,\ud \rho,\nu,\Psi) \\
=  DJ(\bar u,\bar y_0) + DG_1(\bar u,\bar y_0) ^* F_\rho + DG_2(\bar u,\bar y_0) ^* \nu + DG_3(\bar u,\bar y_0) ^* \Psi .
\end{multline}
Let $(\ud \eta, \Psi, p), (\ud \eta', \Psi', p') \in \Lambda$ and suppose that
$\pi\left((\ud \eta, \Psi, p)\right) = \pi \left((\ud \eta', \Psi', p')\right)$.
By lemma~\ref{mult1}, let $(\ud \rho,\nu), (\ud \rho',\nu')$ be such that $(\ud \rho,\nu,\Psi),(\ud \rho',\nu',\Psi') \in \Lambda_R$.
Then $(\nu,\Psi)= (\nu',\Psi')$, and by definition of $\Lambda_R$,
\begin{equation*}
 D_{(u,y_0)} L_R(\bar u,\bar y_0,\ud \rho,\nu,\Psi) =D_{(u,y_0)} L_R(\bar u,\bar y_0,\ud \rho',\nu,\Psi) =0 .
\end{equation*}
Then by \eqref{station}, $ DG_1(\bar u,\bar y_0)^* F_\rho =  DG_1(\bar u,\bar y_0) ^* F_{\rho'}$.
And it is a consequence of lemma~\ref{bigresult}~a) that 
$DG_1(\bar u,\bar y_0)^*$ is injective on $\left( W^{(q),\infty}(\I^\e) \right)^*$.
Then $F_\rho =F_{\rho'}$, and by density of $W^{(q),\infty}(\I^\e) $ in $\prod C \left( \I_i^\e \right)$, we get $\ud \rho= \ud \rho'$.
Together with $\nu=\nu'$, it implies $\ud \eta = \ud \eta'$ and then $(\ud \eta, \Psi, p)= (\ud \eta', \Psi', p')$.
\end{proof}

\paragraph{}
As a corollary, we get a refinement of theorem~\ref{firsteasy}:
\begin{first} \label{first} \hspace{0.cm}
Let $(\bar u,\bar y)$ be a qualified local solution of ($P$). Then $\Lambda$ is nonempty, convex, of finite dimension and compact.
\end{first}

\begin{proof}
Let $\Lambda_\pi := \pi \left( \Lambda \right)$.
By theorem~\ref{firsteasy}, 
$\Lambda$ is nonempty, convex, weakly~$*$ compact and 
$\Lambda_\pi$ is nonempty, convex, of finite dimension and compact
($\pi$ is linear continuous and its values lie in a finite-dimensional vector space).
By lemma~\ref{inj}, $\pi|_{\Lambda} \colon \Lambda  \rightarrow \Lambda_\pi$ is a bijection.
We claim that its inverse
  \begin{equation*}
   \begin{array}{rccc}
m :&  \Lambda_\pi& \longrightarrow & \Lambda   \\
& \left( ([\eta_{i,\tau}])_{\tau,i},\Psi \right)  & \longmapsto &(\ud \eta, \Psi, p) 
 \end{array}
\end{equation*}
 is the restriction of a continuous affine map.
Since $\Lambda = m\left(\Lambda_\pi \right)$, the result follows.
For the claim, using the convexity of both $\Lambda_ \pi$ and $\Lambda$, the linearity of $\pi$ and its injectivity when restricted to $\Lambda$,
we get that $m$ preserves convex combinations of elements from $\Lambda_\pi$.
Thus we can extend it to an affine map on the affine subspace of $\R^{N*}\times \R^{s*}$ spanned by $\Lambda_\pi$.
Since this subspace is of finite dimension, the extension of $m$ is continuous.
\end{proof}

\subsubsection{Second-order conditions on a large critical cone}

Recall that for $\lambda \in \Lambda$, $\J[\lambda]$ has been defined on $\U \times \R^n$ by \eqref{new} or \eqref{old}.

\begin{exlemma} \label{exlemma}
 $\J$ is quadratic w.r.t. $(v,z_0)$ and  affine w.r.t. $\lambda$.
By lemmas~\ref{estimsobo}, \ref{interest} and~\ref{derivative},
$\J[\lambda]$ can be extended continuously to $\V_2 \times \R^n$ for any $\lambda \in \Lambda$.
We obtain the so-called \emph{Hessian of Lagrangian}
\begin{equation} \label{hessian}
 \J \colon \left[ \Lambda \right] \times \V_2 \times \R^n \longrightarrow \R
\end{equation}
which is jointly continuous w.r.t. $\lambda$ and $(v,z_0)$.
\end{exlemma}

\noindent
The critical $L^2$ cone $C_2$ has been defined by \eqref{critical1}-\eqref{critical3}.
Let the \emph{strict critical $L^2$ cone} be the set
\begin{equation*}
 C_2^S := \left\{ (v,z_0) \in C_2 \, : \, g_i'(\bar y)z = 0 \ \text{on } \I_i, \ i=1,\ldots,r \right\},
\end{equation*}
where $z=z[v,z_0] \in \Z_2$.


\begin{second} \label{second}
Let $(\bar u,\bar y)$ be a qualified local solution of ($P$).
 Then for any $(v,z_0) \in C_2^S$, there exists $\lambda \in \Lambda$ such that
\begin{equation}
 \mathcal J[\lambda](v,z_0) \ge 0 .
\end{equation}
\end{second}
The proof is based on the following density lemma, announced in the introduction and proved in the next section:
 \begin{polyhedricity} \label{polyhedricity}
 $C^R_\infty \cap C^S_2$ is dense in $C^S_2$ for the $L^2 \times \R^n$ norm.
\end{polyhedricity}

\begin{proof}[Proof of theorem~\ref{second}]
Let $(v,z_0) \in C^S_2$.
By lemma~\ref{polyhedricity}, there exists a sequence $(v^k,z_0^k) \in C^R_\infty \cap C^S_2$, $k \in \NN$, such that
\begin{equation*}
 (v^k,z_0^k) \longrightarrow (v,z_0).
\end{equation*}
By lemma~\ref{absecond}, there exists a sequence $\lambda^k \in \Lambda$, $k \in \NN$, such that
\begin{equation}\label{limit}
 \mathcal J[\lambda^k](v^k,z_0^k) \ge 0 .
\end{equation}
By theorem~\ref{first}, $\Lambda$ is strongly compact; then there exists $ \lambda \in \Lambda$ such that, up to a subsequence,
\begin{equation*}
 \lambda^k \longrightarrow \lambda.
\end{equation*}
We conclude by passing to the limit in \eqref{limit}, thanks to remark~\ref{exlemma}.
\end{proof}

\subsubsection{A density result} \label{densityresult}

In this section we prove lemma~\ref{polyhedricity}, using lemma~\ref{bigresult}~b).
A result similar to lemma~\ref{polyhedricity} is stated, in the framework of ODEs, as Lemma~5 in \cite{MR2779110}, but the proof given there is wrong.
Indeed, the costates in the optimal control problems of steps a) and c) are actually not of bounded variations and thus the solutions are not essentially bounded.
It has to be highlighted that in lemma~\ref{bigresult}~b) we get a sequence of essentially bounded $v^k$.


\begin{proof}[Proof of lemma~\ref{polyhedricity}]
We define one more cone:
\begin{equation*}
 C^{R+}_\infty = \left\{ (v,z_0) \in C^R_\infty \cap C^S_2 \, : \,
\exists \delta>0 \,:\, g_i'(\bar y)z[v,z_0] =0  \ \text{on } \I_i^{\delta}, \ i=1,\ldots,r \right\},
\end{equation*}
and we show actually that $C^{R+}_\infty$ is dense in $C^S_2$.

To do so, we consider the following two normed vector spaces:
\begin{gather*}
 X_\infty^+ := \left\{ (v,z_0) \in \U \times \R^n \,: \, 
\exists \delta>0 \,:\, g_i'(\bar y)z[v,z_0] =0  \ \text{on } \I_i^{\delta}, \ i=1,\ldots,r \right\},  \\
X_2:= \left\{ (v,z_0) \in \V_2 \times \R^n \,: \, 
 g_i'(\bar y)z[v,z_0] =0  \ \text{on } \I_i, \ i=1,\ldots,r \right\}.
\end{gather*}
Observe that $C_\infty^{R+}$ and $C_2^S$ are defined as the same polyhedral cone by \eqref{critical2}-\eqref{critical3},
respectively in $X_\infty^+$ and $X_2$.
In view of Lemma~1 in \cite{MR2472878}, it is then enough to show that $X_\infty^+$ is dense in $X_2$.

We will need the following lemma, proved in appendix~\ref{noass}:
\begin{density} \label{density}
 Let $\bb_i \in W^{(q_i),2}(\I_i^\e)$ be such that
\begin{equation}
 \bb_i = 0 \ \text{on } \I_i .
\end{equation}
Then there exists $b_i^\delta \in W^{(q_i),\infty}(\I_i^\e)$, $\delta \in (0,\e)$, such that
 $b_i^\delta \xrightarrow[\delta \to 0]{W^{(q_i),2}} \bb_i$ and
\begin{equation}
 b^\delta_i = 0 \ \text{on } \I^\delta_i .
\end{equation}
\end{density}

Going back to the proof of lemma~\ref{polyhedricity},
let $(\vb,\zb_0) \in X_2$
and $\bb:=\left( \A_{2,\zb_0} \vb \right) |_{\I^\e }.$
We consider a sequence $\delta_k \searrow 0$
and for $i=1,\ldots,r$, $b_i^k:= b_i^{\delta_k} \in W^{(q_i),\infty}(\I_i^\e)$ given by lemma~\ref{density}. 
Applying lemma~\ref{bigresult}~b) to $b^k$, we get $v^k$, $k \in \NN$.
We have $(v^k,\zb_0) \in X_\infty^+$ and $(v^k,\zb_0) \longrightarrow (\vb,\zb_0)$. 
The proof is completed.
\end{proof}

\subsection{Sufficient conditions}

We still are under the assumptions \textbf{(A0)}-\textbf{(A4)}.

\begin{legendre}
 A quadratic form $Q$ over a Hilbert space $X$ is a \emph{Legendre form} if it is weakly lower semi-continuous and if it satisfies the following property:
if $x^k \rightharpoonup x$ weakly in $X$ and $Q(x^k) \rightarrow Q(x)$, then $x^k \rightarrow x$ strongly in $X$.
\end{legendre}


\begin{sufficient} \label{sufficient}
 Suppose that for any $(v,z_0) \in C_2$, there exists
$\lambda\in \Lambda$ such that $\mathcal J[\lambda]$ is a Legendre form and
\begin{equation} \label{suffcond}
 \mathcal J[\lambda](v,z_0) > 0 \quad \text{if } (v,z_0) \ne 0 .
\end{equation}
Then $(\ub,\yb)$ is a local solution of ($P$) satisfying the following quadratic growth condition:
there exists $\beta >0$ and $\alpha > 0$ such that
\begin{equation} \label{quadgrowth}
 J(u,y_0) \ge J(\ub,\bar y_0) + \frac{1}{2} \beta \left( \|u- \bar u\|_2 + |y_0 -\bar y_0| \right)^2
\end{equation}
for any trajectory $(u,y)$ feasible for ($P$) and such that $\|u- \bar u\|_\infty + |y_0 -\bar y_0| \le \alpha$.
\end{sufficient}

\begin{huu}
 Let $\lambda =(\ud \eta, \Psi,p) \in \Lambda$. 
The \emph{strengthened Legendre-Clebsch condition}
\begin{equation} \label{clebsch}
 \exists \bar \alpha >0 \, : \, D^2_{uu}H[p](t,\ub_t,\yb_t) \ge \bar \alpha I_m \ \text{for a.a. } t \in  [0,T] 
\end{equation}
is satisfied \textit{iff} $\mathcal J[\lambda]$ is a Legendre form
(it can be porved by combining Theorem~11.6 and Theorem~3.3 in \cite{MR0046590}).
\end{huu}

\begin{proof}[Proof of theorem~\ref{sufficient}]

\noindent \textit{(i)}
 Let us assume that \eqref{suffcond} holds but that \eqref{quadgrowth} does not.
Then there exists a sequence of feasible trajectories $(u^k,y^k)$ such that
\begin{equation} \label{negation}
\left\{ \ba{l}
 (u^k,y_0^k) \xrightarrow[]{L^\infty \times \R^n}(\bar u,\bar y_0),\ (u^k,y_0^k) \ne (\bar u,\bar y_0),\\
 J(u^k,y_0^k) \le J(\bar u,\bar y_0) + o \left( \|u^k- \bar u\|_2 + |y_0^k -\bar y_0| \right)^2.
\ea \right.
\end{equation}
Let $\sigma_k:= \|u^k- \bar u\|_2 + |y_0^k -\bar y_0|$ and
$(v^k,z_0^k):=\sigma_k^{-1}\left( u^k-\ub,y_0^k-\yb_0 \right) \in \U \times \R^n$.
There exists $(\vb,\zb_0)\in \V_2 \times \R^n$ such that, up to a subsequence,
\begin{equation*}
(v^k,z_0^k)\rightharpoonup (\vb,\zb_0) \text{ weakly in } \V_2 \times \R^n . 
\end{equation*}

\bigskip
\noindent
\textit{(ii)}
We claim that  $(\vb,\zb_0) \in C_2$.

Let $z^k:=z[v^k,z_0^k] \in \Y$ and $\zb:=z[\vb,\zb_0] \in \Z_2$.
We derive from the compact embedding $\Z_2 \subset C\left( [0,T];\R^n \right)$ that, up to a subsequence,
\begin{equation} \label{compactinj}
 z^k \rightarrow \zb \text{ in } C\left( [0,T];\R^n \right).
\end{equation}
Moreover, it is classical (see e.g. the proof of Lemma~20 in \cite{MR2421298}) that
\begin{gather}
  J(u^k,y_0^k)= J(\ub,\yb_0) + \sigma_k DJ(\ub,\yb_0)(v^k,z_0^k) + o(\sigma_k)
,\label{have1} \\
  g(y^k)= g(\yb)+\sigma_k g'(\yb)z^k + o(\sigma_k)
 ,\label{have2} \\
 \Phi(y_0^k,y^k_T) = \Phi(\yb_0,\yb_T) + \sigma_k D\Phi(\bar y_0,\bar y_T)(z_0^k,z^k_T) + o(\sigma_k)
.\label{have3}
\end{gather}
It follows that
\begin{gather}
  DJ(\ub,\yb_0)(\vb,\zb_0)  \le 0, \label{part1}\\
\left\{ \ba{ll}
 g_i'(\yb)\zb \le 0 \ \text{on } \I_i & i=1,\ldots,r_1 ,\\
 g_i'(\yb)\zb \le 0 \ \text{on } \I_i\cup \T_i & i=r_1+1,\ldots,r .
\ea \right. \label{part2} \\
 D\Phi(\bar y_0,\bar y_T)(\zb_0,z[\vb,\zb_0]_T) \in T_K \left(\Phi(\bar y_0,\bar y_T)\right) , \label{part3}
\end{gather}
using \eqref{negation} for \eqref{part1} and the fact that $(\ub,\yb)$, $(u^k,y^k)$ are feasible for \eqref{part2} and \eqref{part3}.
By lemma~\ref{coupe}, given $\bar \lambda=(\ud \bar \eta,\bar \Psi,\bar p) \in \Lambda$, we have
\begin{equation*}
  DJ(\ub,\yb_0)(\vb,\zb_0) + \int_{[0,T]}\ud \bar \eta_t g'(\yb_t) + \bar \Psi  D\Phi(\bar y_0,\bar y_T)(\zb_0,\zb_T) = 0.
\end{equation*}
Together with definition~\ref{lagrangemult} and \eqref{part1}-\eqref{part3}, it implies that each of the three terms is null, i.e. $(\vb,\zb_0) \in C_2$.

\bigskip
\noindent
\textit{(iii)}
Then by \eqref{suffcond} there exists $ \bar \lambda \in \Lambda$ such that $\J[\bar \lambda]$ is a Legendre form and
\begin{equation} \label{conclu1}
 0  \le \J[\bar \lambda](\vb,\zb_0) .
\end{equation}
In particular, $\J[\bar \lambda]$ is weakly lower semi continuous. Then
\begin{equation}  \label{conclu2}
 \J[\bar \lambda](\vb,\zb_0) \le \liminf_{k} \J[\bar \lambda](v^k,z_0^k) \le \limsup_{k}\J[\bar \lambda](v^k,z_0^k).
\end{equation}
And we claim that
\begin{equation} \label{conclu3}
 \limsup_{k}\J[\bar \lambda](v^k,z_0^k) \le 0 .
\end{equation}
Indeed, similarly to \eqref{have1}-\eqref{have3}, one can show that, $\bar \lambda$ being a multiplier,
\begin{equation}
 L_R(u^k,y_0^k,\bar \lambda) -  L_R(\ub,\yb_0,\bar \lambda) \label{cool}
=\frac{1}{2} \sigma_k^2 D^2_{(u,y_0)^2}L_R(\ub,\yb_0,\bar \lambda)(v^k,z_0^k)^2 + o(\sigma_k^2).
\end{equation}
Since $L_R(u^k,y_0^k,\bar \lambda) -  L_R(\ub,\yb_0,\bar \lambda) \le J(u^k,y_0^k)-J(\ub,\yb_0)$,
we derive from \eqref{negation}, \eqref{cool} and lemma~\ref{mult2} that
\begin{equation} \label{contradict}
  \J[\bar \lambda](v^k,z_0^k) \le  o(1) .
\end{equation}

\bigskip
\noindent
\textit{(iv)}
We derive from \eqref{conclu1}, \eqref{conclu2} and \eqref{conclu3} that
\begin{equation*}
 \J[\bar \lambda](v^k,z_0^k)\longrightarrow 0 =  \J[\bar \lambda](\vb,\zb_0) .
\end{equation*}
By \eqref{suffcond}, $(\vb,\zb_0) = 0 $, and by definition of a Legendre form, $(v^k,z_0^k)\longrightarrow(\vb,\zb_0)$  strongly in $\V_2 \times \R^n$.
We get a contradiction with the fact that $\|v^k\|_2 + |z^k_0| = 1 $ for all $k$.

\end{proof}

In view of theorems~\ref{second} and~\ref{sufficient} it appears that under an extra assumption, of the type of strict complementarity on the running state constraints,
we can state no-gap second-order optimality conditions.
We denote by $\ri\left( \Lambda \right)$ the relative interior of $\Lambda$ (see Definition~2.16 in \cite{MR1756264}).
\begin{nogap}
Let $(\ub,\yb)$ be a qualified feasible trajectory for $(P)$.
We assume that $C_2^S = C_2$ and that for any $\lambda \in \ri\left( \Lambda \right)$,
the strengthened Legendre-Clebsch condition \eqref{clebsch} holds.
Then $(\ub,\yb)$ is a local solution of ($P$) satisfying the quadratic growth condition \eqref{quadgrowth}
\textnormal{iff}
for any $(v,z_0) \in C_2 \setminus \{0\}$, there exists
$\lambda\in \Lambda$ such that
\begin{equation} \label{finfin}
 \mathcal J[\lambda](v,z_0) > 0  .
\end{equation}
\end{nogap}

\begin{proof}
 Suppose \eqref{finfin} holds for some $\lambda \in \Lambda$; then it holds for some $\lambda \in \ri\left( \Lambda \right)$ too and now $\J[\lambda]$ is a Legendre form.
 By theorem~\ref{sufficient}, there is locally quadratic growth.

Conversely, suppose \eqref{quadgrowth} holds for some $\beta >0$ and let 
\[
 J_\beta (u,y_0):=J(u,y_0)-  \frac{1}{2}\beta \left( \|u- \bar u\|_2 + |y_0 -\bar y_0| \right)^2 .
\]
Then $(\ub,\yb_0)$ is a local solution of the following optimization problem:
\[
\displaystyle \min_{(u,y_0)\in \mathcal U \times \R^n} J_\beta(u,y_0), \quad \text{subject to} \quad 
G_i(u,y_0) \in K_i , \ i=1,2,3.
\]
This problem has the same Lagrange multipliers as the reduced problem (write that the respective Lagrangian is stationary at $(\ub,\yb_0)$),
the same critical cones and its Hessian of Lagrangian is
\[
 \J_\beta[\lambda](v,z_0) =  \J[\lambda](v,z_0) -\beta \left( \|v\|_2 + |z_0| \right)^2 .
\]
Theorem~\ref{second} applied to this problem gives \eqref{finfin}.
\end{proof}

\begin{finalrmk}
A sufficient condition (not necessary \textit{a priori}) to have $C_2^S = C_2$
is the existence of $(\ud  \bar \eta, \bar \Psi,\bar p) \in \Lambda$ such that
\begin{equation*}
 \supp (\ud \bar \eta_i) = \I_i, \ i = 1, \ldots,r .
\end{equation*}

\end{finalrmk}

\appendix
\section{Appendix}

\subsection{Functions of bounded variations} \label{BVfunc}

The main reference here is~\cite{MR1857292}, Section~3.2.
Recall that with the definition of $BV\left([0,T];\R^{n*}\right)$ given at the beginning of section~\ref{Lagrange multipliers},
for $h \in BV\left([0,T];\R^{n*}\right)$ there exist $h_{0_-}, h_{T_+} \in \R^{n*}$
such that \eqref{moinsplus} holds.
\begin{out} \label{out}
Let $h \in BV\left([0,T];\R^{n*}\right)$.
Let $h^l$, $h^r$ be defined for all $t \in [0,T]$ by
\begin{align}
 h^l_t & := h_{0_-} + \ud h\big( [0,t) \big) ,\\
 h^r_t & := h_{0_-} + \ud h\big( [0,t] \big).
\end{align}
Then they are both in the same equivalence class of $h$, $h^l$ is left continuous, $h^r$ is right continuous and, for all $t \in [0,T]$,
\begin{align}
 h^l_t & = h_{T_+}-  \ud h\big( [t,T] \big),  \label{yooo} \\
 h^r_t & = h_{T_+}-  \ud h\big( (t,T] \big).
\end{align}
 \end{out}
 
 \begin{proof}
Theorem~3.28 in~\cite{MR1857292}.
 \end{proof}

The identification between measures and functions of bounded variations that we mention at the beginning of section~\ref{Lagrange multipliers}
relies on the following:

\begin{isombv} \label{isombv}
 The linear map
\begin{equation}
  (c, \mu) \longmapsto \Big( h \colon t \mapsto  c- \mu\left([t,T]\right) \Big)
\end{equation}
is an isomorphism between $\R^{r*} \times \M\left( [0,T];\R^{r*} \right)$ and $BV\left( [0,T];\R^{r*} \right)$, whose inverse is
\begin{equation}
 h \longmapsto \Big( h_{T_+}, \ud h \Big) .
\end{equation}
\end{isombv}

\begin{proof}
Theorem~3.30 in~\cite{MR1857292}.
\end{proof}

Let us now prove lemma~\ref{existadj}:

\begin{proof}[Proof of lemma~\ref{existadj}]
By \eqref{yooo}, a solution in $\cP$ of~\eqref{adjfinal} is any $p \in L^1(0,T;\R^{n*})$ such that, for a.e. $t \in [0,T]$,
\begin{equation}
  p_t  = D_{y_2}\Phi[\Psi](y_0,y_T) + \int_t^T D_y H[p](s,u_s,y_s)\ud s + \int_{[t,T]} \ud \eta_s g'(y_s) .
\end{equation}
 We define $ \Theta \colon L^1(0,T;\R^{n*}) \rightarrow L^1(0,T;\R^{n*})$ by
 \begin{equation}
    \Theta(p)_t := D_{y_2}\Phi[\Psi](y_0,y_T) + \int_t^T D_y H[p](s,u_s,y_s)\ud s + \int_{[t,T]} \ud \eta_s g'(y_s)
 \end{equation}
for a.e. $t \in [0,T]$, and we show that $\Theta$ has a unique fixed point.
Let $C>0$ such that $\| D_y f \|_\infty ,\| D^2_{y,\tau} f \|_\infty \le C$ along $(u,y)$.
\begin{align*}
\left| \Theta(p_1)_t - \Theta(p_2)_t \right|& =  \left|\int_t^T \left( D_y H[p_1](s,u_s,y_s)- D_y H[p_2](s,u_s,y_s) \right) \ud s \right| \\
& \le  C \int_t^T \left[ |p_1(s)-p_2(s) | + \int_s^T |p_1(\theta) - p_2(\theta) | \ud \theta \right] \ud s \\
& =  C \int_t^T \left[ |p_1(s)-p_2(s) | + \int_t^s |p_1(s) - p_2(s) | \ud \theta \right] \ud s \\
& \le  C(1+T) \int_t^T  |p_1(s)-p_2(s) | \ud s .
\end{align*}
We consider the family of equivalent norms on $L^1(0,T;\R^{n*})$
\begin{equation}
 \| v \|_{1,K} := \| t \mapsto e^{-K(T-t)}v(t) \|_1 \quad (K \ge 0).
\end{equation}
\begin{align*}
  \|  \Theta(p_1) - \Theta(p_2) \|_{1,K} & \le  C(1+T) \int_0^T \int_t^T e^{-K(T-t)} |p_1(s)-p_2(s) |\ud s \ud t \\
  & =  C(1+T) \int_0^T e^{-K(T-s)} |p_1(s)-p_2(s)|  \left[ \int_0^s e^{K(t-s)} \ud t \right] \ud s \\
  & \le \frac{C(1+T)}{K} \|  p_1 -p_2 \|_{1,K}.
\end{align*}
For $K$ big enough $\Theta$ is a contraction on $L^1(0,T;\R^{n*})$ for $ \| \cdot \|_{1,K} $;
its unique fixed point is the unique solution of~\eqref{adjfinal}.
\end{proof}

Another useful result is the following integration by parts formula:
\begin{ipp} \label{ipp} 
 Let $h,k \in BV\left( [0,T] \right)$. Then $h^l \in L^1(\ud k)$, $k^r \in L^1(\ud h)$ and
 \begin{equation} \label{ippform}
    \int_{[0,T]} h^l \ud k+ \int_{[0,T]} k^r \ud h = h_{T_+}k_{T_+}- h_{0_-}k_{0_-} .
 \end{equation}
\end{ipp}

 \begin{proof}
Let $\Omega:=\{0\le y\le x\le T \}$. Since $\chi_\Omega \in L^1(\ud h \otimes \ud k)$, we have by Fubini's Theorem (Theorem~7.27 in~\cite{MR1681462})
and lemma~\ref{out} that
$h^l \in L^1(\ud k)$, $k^r \in L^1(\ud h)$ and we can compute $\ud h\otimes \ud k(\Omega)$ in two different ways:
\begin{align*}
 \ud h\otimes \ud k(\Omega) & = \int_{[0,T]} \int_{[y,T]}\ud h_x \ud k_y \\
 & = \int_{[0,T]}\left( h_{T_+}-h^l_y \right) \ud k_y \\
 & = h_{T_+}\left( k_{T_+}-k_{0_-} \right) - \int_{[0,T]}h^l_y \ud k_y , \\
 \ud h\otimes \ud k(\Omega) & = \int_{[0,T]} \int_{[0,x]}\ud k_y \ud h_x \\
 & = \int_{[0,T]} k^r_x \ud h_x - k_{0_-}\left( h_{T_+}-h_{0_-} \right) .
\end{align*}
 \end{proof}

\subsection{The hidden use of assumption 3} \label{ass3}

We use \textbf{(A3)} to prove lemma~\ref{resol} (and then lemma~\ref{bigresult}, and then \ldots)
through the following:

\begin{invers} \label{invers}
Recall that $M_t:= D_{\tilde u}G^{(q)}_{I^{\e_0}(t)}(t,\bar u_t,\bar y_t,\bar u,\bar y) 
\in \R^{|I^{\e_0}(t)|} \times \R^{m*}$.
Then for all $t \in [0,T]$, $M_tM_t^T$ is invertible and $|\left( M_tM_t^T\right) ^{-1}| \le \gamma^{-2}$.
\end{invers}
\begin{proof}
For any $x \in \R^{|I^{\e_0}(t)|}$,
\begin{equation*}
 \left\langle  M_t M_t^T x,x\right\rangle  =|M_t^Tx|^2 \ge \gamma^2 |x|^2.
\end{equation*}
Then $M_t M_t^T x =0 $ implies $x=0$ and the invertibility follows.

Let $y \in \R^{|I^{\e_0}(t)|}$ and $x:=\left( M_tM_t^T\right) ^{-1}y$.
\begin{equation*}
 |y||x| \ge \left\langle y,x\right\rangle  = \left\langle M_tM_t^Tx,x\right\rangle = |M_t^Tx|^2 \ge \gamma^2 |x|^2 .
\end{equation*}
For $y \ne 0$, we have $x \ne 0$; dividing the previous inequality by $|x|$, we get
\begin{equation*}
  \gamma^2 \left|\left( M_tM_t^T\right) ^{-1}y \right| \le |y| .
\end{equation*}
The result follows.
\end{proof}

Before we prove lemma~\ref{resol}, we define the truncation of an integrable function:
\begin{truncation} \label{truncation}
Given any $\phi \in L^s(J)$ ($s \in [1,\infty)$ and $J$ interval),
we will call \emph{truncation of $\phi$} the sequence $\phi^k\in L^\infty(J)$
defined for $k \in \mathbb N$ and a.a. $t \in J$ by
\begin{equation*}
  \phi^k_t := \begin{cases}
        \phi_t  & \text{if }|\phi_t| \le k,\\
	 k \dfrac{\phi_t}{|\phi_t|} & \text{otherwise.} 
          \end{cases}
\end{equation*}
Observe that $ \phi^k \xrightarrow[k \to  \infty ]{L^s} \phi$.
\end{truncation}

\begin{proof}[Proof of lemma~\ref{resol}]
 In the sequel we omit $z_0$ in the notations.

\noindent
\textit{(i)}
Let $v \in \V_s$. We claim that $v$ satisfies
\begin{equation} \label{up}
 M_t v_t + N_t \left( z[v]_t,v,z[v]\right) = h_t \ \text{for a.a. } t \in J_l
\end{equation}
\textit{iff} there exists $w \in L^s(J_l;\R^m)$ such that $(v,w)$ satisfies
\begin{equation} \label{down}
 \left\{ \ba{l}
M_t w_t = 0 , \\
v_t  = M_t^T \left( M_t M_t^T \right)^{-1} \left( h_t - N_t (z[v]_t,v,z[v]) \right) + w_t, 
\ea \right. 
\ \text{for a.a. } t \in J_l.
\end{equation}
Clearly, if $(v,w)$ satisfies \eqref{down}, then $v$ satisfies \eqref{up}.
Conversly, suppose that $v$ satisfies \eqref{up}.
With lemma~\ref{invers} in mind, we define $\alpha \in L^s(J_l;\R^{|I_l|})$ and $w\in L^s(J_l;\R^m)$ by
\begin{gather*}
\alpha := \left( M M^T\right)^{-1}M v ,\\
w:= \left(I_m -M^T \left( M M^T\right)^{-1}M\right)  v .
\end{gather*}
Then
\begin{equation} \label{deri}
 \left\{ \ba{l}
Mw =0 , \\
v = M^T  \alpha + w , 
\ea \right.
\ \text{on } J_l .
\end{equation}
We derive from \eqref{up} and \eqref{deri} that
\begin{equation*}
 M_tM_t^T \alpha_t + N_t \left( z[v]_t,v,z[v]\right) = h_t \ \text{for a.a. } t \in J_l .
\end{equation*}
Using again lemma~\ref{invers} and \eqref{deri}, we get \eqref{down}.

\bigskip
\noindent
\textit{(ii)}
Given $(v,h,w) \in \V_s \times L^s(J_l;\R^{|I_l|}) \times L^s(J_l;\R^m)$,
there exists a unique $\vt \in \V_s$ such that
\begin{equation} \label{wellposed}
 \left\{ \ba{l}
\vt = v \ \text{on } J_0 \cup \cdots \cup J_{l-1} \cup J_{l+1} \cup \cdots \cup J_\kappa , \\
\vt_t  = M_t^T \left( M_t M_t^T \right)^{-1} \left( h_t - N_t (z[\vt]_t,\vt,z[\vt]) \right) + w_t \ \text{for a.a. } t \in J_l,
\ea \right. 
\end{equation}
Indeed, one can define a mapping from $\V_s$ to $\V_s$, using the right-hand side of \eqref{wellposed}.
Then it can be shown, as in the proof of lemma~\ref{existadj}, that this mapping is a contraction for a well-suited norm,
using lemmas~\ref{estimsobo},~\ref{interest} and~\ref{invers}.
The existence and uniqueness follow.
Moreover, a version of the contraction mapping theorem with parameter (see e.g. Th\'eor\`eme~21-5 in \cite{MR0253266})
shows that $\vt$ depends continuously on $(v,h,w)$.

\bigskip
\noindent
\textit{(iii)}
Let us prove a):
let $(\hb,v) \in L^s(J_l;\R^{|I_l|}) \times \V_s$ and let $w:= 0$.
Let $\vt \in \V_s$ be the unique solution of \eqref{wellposed} for $(v,\hb,w)$.
Then $\vt$ is a solution of \eqref{efb} by \textit{(i)}.

\bigskip
\noindent
\textit{(iv)}
Let us prove b):
let $(\hb,\vb) \in L^s(J_l;\R^{|I_l|}) \times \V_s$ as in the statement and let $\wb$ be given by \textit{(i)}.
Then $\vb$ is the unique solution of \eqref{wellposed} for $(\vb,\hb,\wb)$.

Let $(h^k,v^k) \in L^\infty(J_l; \R^{|I_l|}) \times \U$, $k \in \NN$, be such that
$(h^k,v^k)\xrightarrow[]{L^s\times L^s} (\hb,\vb)$ and let $w^k \in L^\infty(J_l;\R^m)$, $k \in \NN$, be the truncation of $\wb$.
It is obvious from definition~\ref{truncation} that
\begin{equation*}
 M_t w^k_t = 0 \ \text{for a.a. } t \in J_l .
\end{equation*}
Let $\vt^k \in \U$ be the unique solution of \eqref{wellposed} for $(v^k,h^k,w^k)$, $k \in \NN$.
Then by uniqueness and continuity in \textit{(ii)},
\begin{equation}
 \vt^k \xrightarrow[]{L^s} \vb.
\end{equation}
And $\vt^k$ is a solution of \eqref{hfb} by \textit{(i)}.
\end{proof}

\subsection{Approximations in $W^{q,2}$} \label{noass}
We will prove in this section lemmas~\ref{raccords} and~\ref{density}.
First we give the statement and the proof of a general result:

\begin{rempl} \label{rempl}
 Let $\xh \in W^{q,2}([0,1])$.
For $j=0,\ldots,q-1$, we denote

\begin{equation}
 \left\{ \ba{l}
\hat \alpha_j := \xh^{(j)}(0) , \\
\hat \beta_j := \xh^{(j)}(1) ,
\ea \right.
\end{equation}
and we consider $\alpha_j^k, \beta_j^k \in \R^q$, $k\in \NN$, such that $(\alpha_j^k ,\beta_j^k) \longrightarrow (\hat \alpha_j,\hat \beta_j)$.
Then there exists $x^k \in W^{q,\infty}([0,1])$, $k \in \NN$, such that $x^k \xrightarrow[]{W^{q,2}} \xh$ and, for $j=0,\ldots,q-1$,
\begin{equation} \label{ifsat}
 \left\{ \ba{l}
 (x^k)^{(j)}(0) = \alpha^k_j, \\
 (x^k)^{(j)}(1) = \beta^k_j.
\ea \right.
\end{equation}

\end{rempl}

\begin{proof}

 Given $u \in L^2([0,1])$, we define $x_u \in W^{q,2}([0,1])$ by
\begin{equation*}
 x_u(t):= \int_0^t \int_0^{s_1} \cdots \int_0^{s_{q-1}} u(s_q) \ud s_q \ud s_{q-1} \cdots \ud s_1 ,\ t \in [0,1].
\end{equation*}
Then $x_u^{(q)} = u$ and, for $j=0,\ldots,q-1$,
\begin{equation*}
 x_u^{(j)}(1) = \gamma_j \ \Longleftrightarrow \ \langle a_j ,u \rangle_{L^2} = \gamma_j
\end{equation*}
where $a_j \in C([0,1])$ is defined by
\begin{equation*}
 a_j(t) := \frac{(1-t)^{q-1-j}}{(q-1-j)!} , \ t \in [0,1].
\end{equation*}
Indeed, a straightforward induction shows that
\[
  x_u^{(j)}(1) = \int_0^t \int_0^{s_{j+1}} \cdots \int_0^{s_{q-1}} u(s_q) \ud s_q \ud s_{q-1} \cdots \ud s_{j+1}.
\]
Then integrations by parts give the expression of the $a_j$.
Note that the $a_j$ ($j=0, \ldots ,q-1$) are linearly independent in $L^2([0,1])$. Then
\begin{equation*}
\ba{rccc}
 A \colon & \R^q&\longrightarrow &L^2([0,1]) \\
 &\begin{pmatrix}
\lambda_0\\
\vdots\\
\lambda_{q-1}
\end{pmatrix} &\longmapsto & 
\displaystyle \sum_{j=0}^{q-1} \lambda_j a_j
\ea
\end{equation*}
is such that $A^* A$ is invertible ($A^*$ is here the adjoint operator). And
\begin{equation} \label{quinze}
  x_u^{(j)}(1) = \gamma_j, \ j = 0 , \ldots,q-1 \ \Longleftrightarrow \ A^*u= (\gamma_0,\ldots,\gamma_{q-1})^T .
\end{equation}

Going back to the lemma, let $\uh := \xh^{(q)} \in L^2([0,1])$.
Observe that
\begin{equation*}
 \xh(t) = \sum_{l=0}^{q-1}  \frac{\hat \alpha_l}{l !}t^l + x_{\uh}(t), \ t \in [0,1],
\end{equation*}
and that $A^* \uh = (\hat \gamma_0,\ldots,\hat \gamma_{q-1})^T$ where
\begin{equation*}
 \hat \gamma_j := \hat \beta_j - \sum_{l=j}^{q-1} \frac{\hat \alpha_l}{(l-j)!} , \ j = 0,\ldots,q-1 .
\end{equation*}
Then we consider, for $k \in \NN$, the truncation (definition~\ref{truncation}) $\uh^k \in L^\infty([0,1])$ of $\uh$, and
\begin{gather}
 \gamma_j^k := \beta_j^k - \sum_{l=j}^{q-1} \frac{\alpha_l^k}{(l-j)!} , \ j = 0,\ldots,q-1 , \label{huit}\\
\gamma^k := (\gamma_0^k,\ldots,\gamma_{q-1}^k)^T  , \nonumber\\
u^k:= \uh^k + A (A^*A)^{-1} \left( \gamma^k - A^*\uh^k \right)  ,\nonumber \\
x^k(t) := \sum_{l=0}^{q-1}  \frac{\alpha^k_l}{l !}t^l + x_{u^k}(t) , \ t \in [0,1] . \label{unvg}
\end{gather}
It is clear that $u^k \in L^\infty([0,1])$ (by definition of $A$); then $x^k \in W^{q,\infty}([0,T])$.
Since $A^*u^k = \gamma^k$ and in view of \eqref{quinze}, \eqref{huit} and \eqref{unvg}, \eqref{ifsat} is satisfied.
Finally, $\gamma^k_j \longrightarrow \hat \gamma_j$ ($j=0,\ldots,q-1$);
then $\gamma^k \longrightarrow A^* \uh$ and $u^k \longrightarrow \uh$.

\end{proof}

%
%
%

We can also prove the following:

\begin{rempl2} \label{rempl2}
 Let $\xh \in W^{q,2}([0,1])$ be such that $\xh^{(j)}(0)=0$ for $j=0,\ldots, q-1$.
Then there exists $x^\delta \in W^{q,\infty}([0,1])$ for $\delta >0$ such that $x^\delta \xrightarrow[\delta \to 0]{W^{q,2}} \xh$ and
\begin{equation}
x^\delta = 0 \ \text{on } [0,\delta] .
\end{equation}

\end{rempl2}

\begin{proof}
 We consider $u^\delta \in L^\infty([0,1])$, $\delta >0$, such that $u^\delta = 0$ on $[0,\delta]$ and
 $u^\delta \xrightarrow[\delta \to 0]{L^2}\uh :=  \xh^{(q)}$.
 Then we define $x^\delta:=x_{u^\delta}$ (see the previous proof).
\end{proof}

Now the proof of lemma~\ref{density} is straightforward.

\begin{proof}[Proof of lemma~\ref{density}]
 We observe that $\bb_i= 0$ on $\I_i$ implies that $\bb_i^{(j)}=0$ at the end points of $\I_i$ for $j=0,\ldots,q_i-1$
(note that with the definition \eqref{bound}, if one component of $\I_i$ is a singleton, then $q_i=1$).
Then the conclusion follows with lemma~\ref{rempl2} applied on each component of $\I_i^\e \setminus \I_i$.  
\end{proof}

Finally, we use lemma~\ref{rempl} to prove lemma~\ref{raccords}.

\begin{proof}[Proof of lemma~\ref{raccords}]
In the sequel we omit $z_0$ in the notations.
We define a \emph{connection in $W^{q,\infty}$} between $\psi_1$ at $t_1$ and $\psi_2$ at $t_2$
as any $\psi \in W^{q,\infty}([t_1,t_2])$ such that
\begin{equation*}
  \left\{ \ba{l}
\psi^{(j)}(t_1) = \psi_1^{(j)}(t_1) , \\
\psi^{(j)}(t_2) = \psi_2^{(j)}(t_2) ,
\ea \right.
\ j=0,\ldots,q-1 .
\end{equation*}

\noindent a)
We define $\bt_i$ on $[0,t_0]$ by $\bt_i := g_i'(\yb)z[v]$, $\ i = 1, \ldots,r$.
We need to explain how we define $\bt_i$ on $(t_0,T]$, using $\bb_i$ and connections,
to have $\bt_i \in W^{q_i,s}([0,T])$ and $\bt_i=\bb_i$ on each component of $\I_i^\e \cap (t_0,T]$.
The construction is slightly different whether $t_0 \in \I_i^\e$ or not, i.e. whether $i \in I^\e_{t_0}$ or not.
Note that by definition of $\e_0$ and of $t_0$, $I^\e_t$ is constant for $t$ in a neighbourhood of $t_0$.
We now distinguish the 2 cases just mentioned:
\begin{enumerate}
 \item $i \in I^\e_{t_0}$:
We denote by $[t_1,t_2]$ the connected component of $\I_i^{\e}$ such that $t_0 \in (t_1,t_2)$.
We derive from \eqref{hypa1} that $\bt_i= \bb_i$ on $[t_1,t_0]$.
Then we define $\bt_i : = \bb_i $ on $(t_0,t_2]$.

If $\I_i^{\e}$ has another component in $(t_2,T]$, we denote the first one by $[t_1',t_2']$.
Let $\psi$ be a connection in $W^{q_i,\infty}$ between $\bt_i$ at $t_2$ to $\bb_i$ at $t_1'$.
We define $\bt_i:= \psi$ on $(t_2,t_1')$, $\bt_i:=\bb_i$ on $[t_1',t_2']$, and so forth on $(t_2',T]$.

If $\I_i^{\e}$ has no more component, we define $\bt_i$ on what is left as a connection in $W^{q_i,\infty}$ between $\bb_i$
and $g_i'(\yb)z[v]$ at $T$.

 \item $i \not \in I^\e_{t_0}$: 
If $\I_i^{\e}$ has a component in $[t_0,T]$, we denote the first one by $[t_1,t_2]$.
Note that $t_1 - t_0 \ge \e_0 -\e >0$.
We consider a connection in $W^{q_i,\infty}$ between $\bt_i$ at $t_0$ and $\bb_i$ at $t_1$ and we continue as in 1.

If $\I_i^{\e}$ has no component in $[t_0,T]$, we do as in 1.
\end{enumerate}

\noindent b)
For all $k \in \NN$, we apply a) to $(b^k,v^k)$ and we get $\bt^k$. We just need to explain how we can get, for $\ i = 1, \ldots,r $,
\begin{equation*}
 \bt^k_i \xrightarrow[k \to \infty]{W^{q_i,2}} g_i'(\yb)z[\vb].
\end{equation*}
By construction we have
\begin{equation*}
 \ba{ll}
\text{on } [0,t_0],& \bt^k_i =g_i'(\yb)z[v^k] \longrightarrow g_i'(\yb)z[\vb] ,\\
\text{on } \I_i^\e,&\bt^k_i = b^k_i \longrightarrow  \bb_i = g_i'(\yb)z[\vb] .
\ea
\end{equation*}
Then it is enough to show that every connection which appears when we apply a) to $(b^k,v^k)$,
for example $\psi_i^k \in W^{q_i,\infty}([t_1,t_2])$, can be chosen in such a way that
\begin{equation*}
 \psi_i^k \longrightarrow g_i'(\yb)z[\vb]  \ \text{on } [t_1,t_2].
\end{equation*}
This is possible by lemma~\ref{rempl}.
\end{proof}


\bibliographystyle{abbrv}
\bibliography{integraleqns}

\end{document}